\documentclass{article}

\usepackage[sort&compress, round]{natbib}
\usepackage{framed}

\usepackage{amsmath, amssymb, amsthm, amscd, verbatim}

\usepackage[usenames,dvipsnames]{color}

\usepackage{tikz}

\allowdisplaybreaks

\newcommand{\R}  {{\mathbb R}}
\newcommand{\N}  {{\mathbb N}}
\newcommand{\K}  {{\mathbb K}}
\newcommand{\eps}{\varepsilon}

\newcommand{\bx}{{\mathbf x}}
\newcommand{\by}{{\mathbf y}}
\newcommand{\bg}{{\boldsymbol{\gamma}}}
\newcommand{\br}{{\boldsymbol{\rho}}}
\newcommand{\be}{{\boldsymbol{\eta}}}
\newcommand{\bxi}{{\boldsymbol{\xi}}}
\newcommand{\bz}{{\boldsymbol{\zeta}}}
\newcommand{\bn}{{\boldsymbol{0}}}
\newcommand{\bsa}{\boldsymbol{a}}

\newcommand{\B}{{\mathfrak B}}
\newcommand{\Z}{{\mathfrak Z}}

\newcommand{\PP} {\mathcal{P}}
\newcommand{\U}  {\mathcal{U}}
\newcommand{\V}  {\mathcal{V}}
\newcommand{\W}  {\mathcal{W}}
\newcommand{\M}  {\mathcal{M}}
\newcommand{\mN}  {\mathcal{N}}
\newcommand{\SSS}{\mathcal{S}}
\newcommand{\A}{\mathcal{A}}

\newcommand{\EE}{\mathbb{E}}

\newcommand{\lo}{\downarrow}
\newcommand{\up}{\uparrow}

\DeclareMathOperator{\decay}{dec}
\DeclareMathOperator{\err}{{\rm e}}

\theoremstyle{plain}
\newtheorem{lemma}{Lemma}
\newtheorem{theo}{Theorem}
\newtheorem{corollary}{Corollary}
\newtheorem{prop}{Proposition}

\theoremstyle{definition}
\newtheorem{defi}{Definition}
\newtheorem{rem}{Remark}
\newtheorem{exmp}{Example}

\title{Embeddings of Reproducing Kernel Hilbert Spaces 
with General Weights}

\author{Michael Gnewuch, Peter Kritzer, Klaus Ritter}

\date{April 28, 2026}

\begin{document}

\maketitle

\begin{abstract}
We study embeddings between reproducing kernel Hilbert
spaces $H(K)$ of functions of $d \in \N \cup \{\infty\}$ variables.
The kernels $K$ are superpositions of weighted finite
tensor products of a fixed univariate kernel.
The basic idea for the embeddings
is to compensate a change of the univariate kernel 
by a suitable transformation of the weights. For the proofs
we employ ($d \in \N$) and develop ($d = \infty$) a discrete
calculus on the cone of all weights, where completely monotone
weights play a particular role.

We sketch how to apply the embedding results to 
computational problems, as, e.g., numerical integration or 
function recovery.
\end{abstract}




\section{Introduction} 

Our goal is to study computational problems on Hilbert
spaces $H$ of functions that depend on $d$ variables;
the parameter $d$ is also referred to as the \emph{dimension} of the 
problem. A typical situation is the approximation of a continuous 
linear \emph{solution operator} defined on $H$ and mapping to 
a Banach space $G$,
\[
 S \colon H \to G,
\]
by suitable algorithms. 
Here we interpret $S$ as an operator that, in a given computational 
problem, maps an input function $f\in H$ to the corresponding 
(exact)
solution $g\in G$. The goal is to approximate this solution, and the 
approximation error is commonly measured in the norm of $G$.
Typical examples of such a setting are numerical integration, 
function recovery, or well-posed linear operator equations.

This paper is concerned with the case where $d$ is a 
potentially huge positive integer, 
or where $d$ is infinite. In the latter case, the elements of the 
space $H$ depend on infinitely (but countably) many variables. 
Such a situation naturally occurs
in models that are based on infinite sequences 
of independent random variables, as for example in 
the context of
PDEs with random coefficients, see, among many others, 
\citet{CD2015},
\citet{KN16}, and \citet[Chap.~4]{ABW22}.

A standing assumption in our paper is that the function spaces $H$
are \emph{reproducing kernel Hilbert spaces}, and that they are 
\emph{weighted} in the sense of \citet{SW98}.
Let us explain this in more detail. Let $\U_d$ be the power set of 
$\{1,\ldots,d\}$ for finite $d$, and the set of all finite subsets 
of the positive integers if $d=\infty$. We assign non-negative real 
numbers $\gamma_u$ to each set $u\in \U_d$
of variable indices, indicating the 
joint influence of the variables $x_j$ with $j \in u$
on the functions $f \in H$.
A weight with a larger value corresponds to 
a stronger influence,
a weight with a smaller value corresponds to 
a weaker influence.
This setting is also motivated by applications 
beyond random PDEs, as for example in financial 
engineering, where it may frequently be the case that the variables 
$x_j$ have diminishing influence with growing $j$.

The major part of this paper is devoted to the case where $d=\infty$,
simply because this case is more elaborate and technically more 
demanding than the case of finite $d$. We may view the case 
$d=\infty$ as the limiting case of tractability analysis, which,
roughly speaking, deals with finite $d$ tending to $\infty$.
For tractability analysis we refer to the trilogy 
\citet{NW08,NW10,NW12}.

We assume that the kernels $K$ of the spaces 
$H = H(K)$ considered here 
are superpositions of weighted tensor product kernels.
More precisely, $K$ takes the form 
\begin{equation}\label{eq:kernel_superp}
   K (\bx, \by) := \sum_{u \in \U_d} \gamma_u k_u(\bx, \by)
\end{equation}
with weights
\[
\bg := (\gamma_u)_{u \in \U_d}
\]
and finite tensor products
\[
k_u (\bx,\by) := \prod_{j \in u} k(x_j,y_j)
\]
of a given kernel $k$ over an arbitrary domain $D$. 
A suitably chosen domain $\Z \subseteq D^\N$
and $\bx,\by \in \Z$
may be needed to ensure the convergence in
\eqref{eq:kernel_superp} in the case $d = \infty$.

Considering the elements of the Hilbert space $H(k)$ with 
reproducing kernel $k$ as univariate functions, 
the elements of the Hilbert spaces $H(k_u)$ are 
$|u|$-variate functions, since they only depend 
on the variables with indices in $u$. 
We add that finite tensor products are often employed
in the analysis of high-, yet finite-dimensional computational
problems.

In many cases in the literature, authors impose a certain 
structure on the weights $\gamma_u$ for $u\in \U_d$. The 
most commonly considered case is that of product weights, which are 
defined in terms of a sequence $(\gamma_j)_{j\ge 1}$ of non-negative 
real numbers, and for which 
\[
\gamma_u := \prod_{j\in u} \gamma_j
\]
for every $u\in \U_d$ (with the convention that 
$\gamma_\emptyset =1$).
However, in recent years it has turned out that for several 
applications of interest, including the analysis of the 
aforementioned PDEs with random coefficients, one needs to allow for
more general weights that lack the product structure, in particular 
POD (product and order dependent), or the even more involved SPOD 
(smoothness-driven, product and order dependent) weights, 
as used, e.g., in \citet{HHKKS24}; see 
Section~\ref{SEC:Part_Weights} 
for the definition of POD and finite-order weights.

It is known that structural
properties of the function space $H(K)$,
for finite or infinite $d$, may substantially facilitate the error 
analysis for particular types of algorithms. 
A key feature for the error analysis is the orthogonal
decomposition of $H(K)$ that is induced by the representation
\eqref{eq:kernel_superp} of the kernel $K$, if 
$1 \not\in H(k)$.
To give an example, in the analysis of 
randomized algorithms, one frequently would like to know about the 
variance of the algorithm, for which it is 
most convenient if \eqref{eq:kernel_superp} induces
an ANOVA (an acronym for ``Analysis of Variance'')
decomposition of the functions in $H(K)$. 
To provide another example, for deterministic algorithms, often an 
anchored decomposition is favorable. We refer to \citet{BG14} and 
\citet{DG12}, as well as the references therein for 
further details. Moreover, note that anchored decompositions
also play an important role in the construction of 
multivariate decomposition methods (originally named changing 
dimension algorithms) as introduced in \citet{KSWW10a}, and studied 
further in several follow-up papers. 

Hence, we may be in a situation where we would like to analyze a 
certain (type of) algorithm in a space $H(K)$ 
that does \emph{not} allow for an effective error analysis. 
In such a case, we are interested in a way to transfer to a 
different function space where we can derive a meaningful error 
bound that also yields a useful
bound when going back to the original space. This 
can be done by finding suitable embeddings between weighted 
function spaces (in the sense of Sloan and
Wo\'zniakowski). 

This approach has been developed for the case of 
tensor products of weighted Hilbert spaces
in several papers, see, e.g., 
\citet{HR13} and \citet{GHHR17}. It is known that, in terms of the 
weights, the case of a tensor product space $H(K)$
exactly corresponds to the 
assumption of the weights being of product form,
where $K$ given by \eqref{eq:kernel_superp} satisfies
\begin{equation}\label{g17}
K(\bx,\by) = \prod_{j=1}^d (1 + \gamma_j k(x_j,y_j))
\end{equation}
in both cases $d\in\N$ and $d = \infty$. 

In the present paper, 
motivated by applications where one 
does not have product weights, 
but, e.g., POD or other types of weights, we try to 
establish embedding results that cover general weights with as 
few additional assumptions 
as possible, thereby moving away from the tensor product 
assumption for the function spaces $H(K)$ under consideration.

Embeddings of function spaces relying on more general classes 
of weights, but with specific multivariate building blocks,
like Sobolev spaces  of mixed smoothness, 
have already been studied in the literature. See,
e.g., 
\citet{Hefter20161}, \citet{Hinrichs2015}, \citet{GHHRW2017}, 
\citet{GKS2022}.
To our best knowledge, the present paper is the first to address 
embeddings for 
general weights and general tensor product kernels $k_u$.

Let us now give a short outlook on the paper. 
In Section~\ref{sec:weights} we study 
the set $\M_d$ of completely monotone weights $\bg$,
where the latter are considered as set functions,
as well as the \emph{sum operator} $T_{d, C}^\up$, 
given by $T_{d, C}^\up \bg := \bg^\up$ with
\[
\gamma^\up_u :=
\sum_{\substack{v\in \U_d \\ u\subseteq v}}
C^{2\left|v\right|}\gamma_v
\]
for every $u \in \U_d$. Here $C$ denotes a fixed positive 
constant, and \begin{equation}\label{g11}
\sum_{v\in \U_d} C^{2\left|v\right|}\gamma_v < \infty
\end{equation}
is required to hold in the case $d = \infty$. As the main result
for $d=\infty$, we determine the range and the inverse $T^\lo_{d, C}$
of $T^\up_{d, C}$, see Theorem~\ref{t1b};
here the range is a proper subset of $\M_d$, defined by an 
additional $\sigma$-continuity property.
For $d \in \N$ we present a known result concerning the 
range and the inverse in Theorem~\ref{t1a}, see 
\citet{Aigner79} and \citet{Matus94}; 
here the range is equal to $\M_d$.

In Section~\ref{SEC:Part_Weights} we investigate
particular classes of weights, namely product, POD, and 
finite-order weights. First of all, we study the complete 
monotonicity and the summability \eqref{g11} for these kinds of 
weights. Furthermore, we study invariance properties of the 
aforementioned classes of weights with respect to the mappings 
$T^\up_{d, C}$ and $T^\lo_{d, C}$. For the
applications we have in mind, it is crucial to find out whether 
the decay of the weights is invariant under $T_{d, C}^\up$ and 
$T^\lo_{d, C}$, respectively. 
This is because the decay of the weights is intimately related to 
the decay of the minimal errors for the 
computational problems we are aiming at,
and may, 
at least implicitly, occur in the corresponding error bounds, 
which can already be seen, e.g., in \citet{HMNR10},
\citet{KSWW10a}, and \citet{PW11}. Results regarding invariance 
of the decay are given in 
Proposition~\ref{le:decay_POD}, Corollary~\ref{le:decay_prod}, and
Proposition~\ref{le:decay_FO} for POD, product, and finite-order
weights, respectively. 

After having studied the fundamental properties of weights and their 
transformations via $T^\up_{d, C}$ and $T^\lo_{d, C}$
in Sections~\ref{sec:weights} and \ref{SEC:Part_Weights}, 
we formally introduce the function spaces under consideration in 
this paper and the basic embedding results
in Section~\ref{Sec:Embeddings}.

In addition to $k$ and $\bg$, which define $K$ via
\eqref{eq:kernel_superp},
we first of all consider another reproducing kernel
$k^\up$ such that $k$ and $k^\up$ have a common 
domain and
\[
H(k) \subseteq H(1+k^\up).
\]
We add that the stronger assumption
$H(k) \subseteq H(k^\up)$ is typically not satisfied in the 
applications to the computational problems we have
in mind.

Secondly, we consider the weights $\bg^\up :=  
T^\up_{d, C} \bg$, where $C$  may be 
chosen as an arbitrary upper bound on the norm of the embedding of 
$H(k)$ into $H(1+k^\up)$. 
In this way we obtain a pair of reproducing kernels of the 
same form, namely $K$ over the domain $\Z$ and 
analogously $K^\up$ with $k$ 
and $\bg$ replaced by $k^\up$ and $\bg^\up$, respectively, over the
domain $\Z^\up$. We show that the domains
satisfy $\Z^\up \subseteq \Z$ and that
the restriction
\[
f \mapsto f|_{\Z^\up}
\]
defines a bounded linear operator from 
$H(K)$ to $H(K^\up)$ of
norm at most one.
See Theorem~\ref{Lemma4} for the case $d= \infty$ and
Theorem~\ref{Lemma1} for the case $d \in \N$; 
in the latter case we have $\Z^\up = \Z$, implying that
the restriction is actually the identical embedding.

Section~\ref{Sec:Application} shows how the findings on the 
embeddings of function spaces with general 
weights can be used when one 
would like to derive error bounds for a given computational problem 
on a Hilbert space $H(K)$ with reproducing kernel $K$ of the 
form \eqref{eq:kernel_superp} by switching back and forth to an 
RKHS with a similar, but more favorable structure.
This discussion is held in a setting as general as possible, i.e., 
without specifying the computational problem; 
the essential assumption is that the solution operator $S$ and 
all algorithms under consideration are bounded linear operators. 
The sum operator $T^\up_{d, C}$ and a straightforward 
application of the results from Section~\ref{Sec:Embeddings} may be 
used to obtain upper error bounds. Likewise, the inverse 
$T^\lo_{d, C}$ of the sum operator may be used to obtain
lower error bounds. More precisely, if $\bg \in \M_d$ then
$T^\lo_{d, C}$ for finite $d$ or the 
canonical extension of $T^\lo_{d, C}$ to $\M_d$ for $d=\infty$
may be used directly.
Otherwise, the weights $\bg$ are substituted by completely monotone
weights $\bg_\ast \leq \bg$ in the first place.
See Corollary~\ref{c1} for the upper bounds and 
Corollary~\ref{c2} for the lower bounds.

An application to concrete computational problems of the theory 
developed, such as, e.g., $L^p$-approximation in the case 
$d=\infty$, is deferred to a follow-up paper.

The paper is rounded off by an appendix containing 
background on completely monotone weights and related concepts.

\section{General Weights}\label{sec:weights}

\subsection{Notation}\label{sec:notation}

Let $d\in \N \cup \{\infty\}$.
For $d\in\N$ we denote the set $\{1,\ldots,d\}$ by $[d]$, and for 
$d=\infty$ we denote by $[d]$ the whole set $\N$. We put 
\[
\U_{\infty} := \{u\subset \N : \text{$u$ finite} 
\} 
\] 
and 
\[
\U_d := \{u \subset \N : u \subseteq [d] \}
\]
for $d \in \N$. The size of a finite set $u$ is denoted by $|u|$. 

Let $\bg, \bg^\prime, \bg^{(1)}, \ldots \colon \U_d \to \R $ 
denote families of real numbers, indexed by $\U_d$. We write $\bg \geq
\bg^\prime$ if $\gamma_u \geq \gamma^\prime_u$ for every $u \in
\U_d$, and $\lim_{n \to \infty} \bg^{(n)} = \bg$ 
if $\lim_{n \to \infty} \left(\bg^{(n)}\right)_u = \gamma_u$ for 
every $u \in \U_d$. Moreover, the family $\bg$ that is constant 
and equal to zero is denoted by $\bn$.

Let $\V$ be a countably infinite index set. For a 
family $\bsa=(a_v)_{v\in\mathcal{V}}$ of non-negative real
numbers the decay is defined by
\[
\decay\left(\bsa\right)
:=\sup\left\{\tau>0:
\sum_{v\in\V} a_v^{1/\tau} <\infty \right\},
\]
with the convention $\sup\emptyset = 0$. For instance,
we have 
$\decay\left(\bsa\right) = \lambda$ for $\V := \N$ and
$a_v:= v^{-\lambda}$ with $\lambda\geq 0$.

We use the term ``weights'' and the notation 
$\bg=\left(\gamma_u\right)_{u\in \U_d}$
for any family of non-negative real numbers,
indexed by $\U_d$. Moreover, $\W_d$ denotes
the set of all weights, i.e.,
\[
\W_d := \{ \bg \colon \U_d \to \R : \bg \geq \bn \} \subset \R^{\U_d}.
\]

\subsection{Complete Monotonicity}

We define the difference operators $\Delta_v$ on $\R^{\U_d}$ by
\[
\left(\Delta_v \bg \right)_u := 
\sum_{w\subseteq v} (-1)^{|w|} \gamma_{u\cup w}
\]
for $\bg \in \R^{\U_d}$ 
and $u,v \in \U_d$.
Note that $\Delta_\emptyset \bg = \bg$.

\begin{lemma}\label{l1}
For $\bg \in \R^{\U_d}$ and $u,v \in \U_d$ we have 
\[
u \cap v \neq \emptyset \quad \Rightarrow \quad
\left(\Delta_v \bg \right)_u = 0.
\]
\end{lemma}

\begin{proof}
In any case 
\[
\left(\Delta_v \bg \right)_u =
\sum_{w_1\subseteq v \cap u}\ \sum_{w_2\subseteq v \setminus u} 
(-1)^{|w_1|+|w_2|} \gamma_{u\cup w_2}.
\]
Moreover,
\[
\sum_{w_1\subseteq v \cap u} (-1)^{|w_1|} = 0
\]
if $v \cap u \neq \emptyset$.
\end{proof}

Let us consider differences $\Delta_v \bg$ of orders one and
two, i.e., with $|v|=1$ and $|v|=2$. 
If
$v=\{s\}$ for some $s\in [d]$, then
\[
\left(\Delta_v \bg \right)_u = \gamma_u - \gamma_{u\cup \{s\}}
\]
for every $u \in \U_d$.
If $v=\{s_1,s_2\}$ for some distinct 
$s_1,s_2 \in [d]$, then 
\begin{align*}
 \left(\Delta_v \bg \right)_u 
 &= \gamma_u - \gamma_{u\cup \{s_1\}} 
 - \gamma_{u \cup \{s_2\}} 
 + \gamma_{u \cup \{s_1,s_2\}}\nonumber\\
 &= \gamma_u - \gamma_{u\cup \{s_1\}} 
 -  \left(\gamma_{u \cup \{s_2\}} - 
\gamma_{u \cup \{s_1,s_2\}}\right)\\
&= \left( \Delta_{\{s_1\}} \bg \right)_u -   
\left( \Delta_{\{s_1\}} \bg \right)_{u\cup \{s_2\}}
\end{align*}
for every $u \in \U_d$.
The following lemma, which
gives a recursion formula for the 
operators $\Delta_v$,
illustrates that we can interpret all the expressions 
$\left(\Delta_v \bg \right)_u$ with $v \neq \emptyset$
as iterated differences.

\begin{lemma}\label{l2}
For $\bg \in \R^{\U_d}$, $u,v \in \U_d$, and $s \in [d]$ we have
\[
\left(\Delta_{v \cup \{s\}} \bg\right)_u = 
\left(\Delta_v \bg\right)_u  - 
\left(\Delta_v \bg\right)_{u \cup \{s\}}.
\]
\end{lemma}

\begin{proof}
If $s \in v$ then
$\left(\Delta_v \bg\right)_{u \cup \{s\}} = 0$ by Lemma~\ref{l1},
which yields the claim. If $s \not\in v$ then
\[
\left( \Delta_{v \cup \{s\}} \bg\right)_u =
\sum_{w \subseteq v} (-1)^{|w|} 
\left( \gamma_{u \cup w} - \gamma_{u \cup w \cup \{s\}} 
\right)
=
\left( \Delta_v \bg\right)_u - 
\left( \Delta_v \bg\right)_{u\cup \{s\}}
\]
as claimed.
\end{proof}

Next, we show an auxiliary result regarding the composition 
of difference operators.

\begin{lemma}\label{l2a}
We have $(\Delta_v \circ \Delta_w) \bg = \Delta_{v \cup w} \bg$
for all $v,w\in \U_d$ and all $\bg \in \R^{\U_d}$.
\end{lemma}

\begin{proof}
Lemma~\ref{l2} yields 
$\Delta_{v \cup \{s\}} \bg = \Delta_{\{s\}} ( \Delta_v \bg)$
for every $s \in [d]$ and every $v\in \U_d$. 
The statement of the present lemma trivially
holds true for $v = \emptyset$, and inductively we obtain
\[
\Delta_{v \cup \{s\} \cup w} \bg =
\Delta_{\{s\}} (\Delta_{v \cup w} \bg) =
\Delta_{\{s\}} (\Delta_v (\Delta_w \bg)) =
\Delta_{v \cup \{s\}} (\Delta_w \bg),
\]
which completes the proof.
\end{proof}

\begin{defi}\label{def:comp_mon_dec}
The weights $\bg \in \W_d$ are called 
\emph{completely monotone} 
if
\[
\left(\Delta_v \bg \right)_u \ge 0
\] 
for all $u,v\in \U_d$. We denote by 
\[
\M_d := \{ \bg \in \W_d : 
\text{$\Delta_v \bg \geq \bn$ for every $v \in \U_d$}\}
\]
the set of all completely monotone weights in $\W_d$.
\end{defi}

The notion of complete monotonicity from 
Definition~\ref{def:comp_mon_dec} is a particular instance of a 
general concept 
for set functions,  see \citet{Kurtz74}.
Moreover, for $d \in \N$ complete monotonicity of  
functions on the set $\U_d$, or equivalently on the discrete unit
cube $\{0,1\}^d$, is analogous
to complete monotonicity of functions on the continuous
unit cube $[0,1]^d$. See Appendices~\ref{a1} and \ref{a2}
for details.
The study of complete monotonicity of weights in the context 
of weighted function spaces (in the sense of Sloan and
Wo\'zniakowski) seems to be new.

Completely monotone weights satisfy the following
monotonicity conditions, 
the first of which has 
already been employed
in the literature on high-dimensional integration, 
see, e.g.,~\citet[Eqn.~(5.37)]{DKS13}.

\begin{lemma}\label{cmd:claim_i}
For $\bg \in \M_d$ and $v, w \in \U_d$ we have
\[
v \subset w \quad \Rightarrow \quad
\left(\gamma_v \geq \gamma_w \ \wedge \ \Delta_v \bg \ge \Delta_w \bg
\right).
\]
\end{lemma}

\begin{proof}
Due to Lemma~\ref{l2a} 
it suffices to consider the case $w = v \cup \{s\}$ with 
$s \not\in v$. Using Lemma~\ref{l2}
and $\bg \in \M_d$ we obtain
\[
\gamma_v = \left( \Delta_\emptyset \bg\right)_v
= \left( \Delta_{\{s\}} \bg\right)_v
+ \left( \Delta_\emptyset \bg\right)_w
\geq
\left( \Delta_\emptyset \bg\right)_w = \gamma_w
\]
as well as
\[
\left( \Delta_v \bg\right)_{u} =
\left( \Delta_w \bg\right)_{u} +
\left( \Delta_{v} \bg\right)_{u \cup \{s\}} \geq
\left( \Delta_w \bg\right)_{u}
\]
for any $u \in \U_d$.
\end{proof}

Let $\max \emptyset := 0$ by convention.
For $d = \infty$ and $u,v \in \U_d$ we define $u \leq v$ 
if $\max u \leq \max v$ to turn
$\U_d$ into a directed set, and we put
\[
\mN_d:= \{\bg \in \M_d \colon \lim_u \gamma_u = 0\}.
\]
For $d \in \N$ we put $\mN_d := \M_d$.

\begin{lemma}\label{lem:diff_op}
For $v\in \U_d$ the difference operator $\Delta_v$ 
satisfies $\Delta_v (\M_d) \subseteq \M_d$ as well as
$\Delta_v (\mN_d) \subseteq \mN_d$.
\end{lemma}

\begin{proof}
Let $v\in \U_d$. Due to Lemma~\ref{l2a} and the definition of 
$\M_d$, we obtain for $\bg \in \M_d$ and each $w\in \U_d$ that
$\Delta_w(\Delta_v \bg) = \Delta_{v \cup w} \bg \ge \bn$, i.e., 
$\Delta_v \bg \in \M_d$. \\
Now consider $\bg \in \mN_d$ in the case $d = \infty$.
Since $\mN_d \subseteq \M_d$, we know
that $\Delta_v \bg \in \M_d$. Furthermore, Lemma~\ref{cmd:claim_i} yields 
$\Delta_v \bg \le \Delta_\emptyset \bg = \bg$. Due to the definition of 
$\mN_d$, these facts imply that $\Delta_v \bg \in \mN_d$.
\end{proof}

\subsection{The Sum Operator and Its Inverse}
\label{Subsec:Sum_Op_Inv}

For any $C>0$ we consider the set
\[
\SSS_{d,C} := \left\{ \bg \in \W_d : 
\sum_{v \in \U_d} C^{2|v|} \gamma_v < \infty \right\}
\]
of weights and the sum operator 
\[
T^{\up}_{d,C} \colon \SSS_{d,C} \to \W_d
\]
given by
\[
\left(T^\up_{d,C} \bg\right)_u :=
\sum_{\substack{v\in \U_d \\ u\subseteq v}}
C^{2\left|v\right|}\gamma_v.
\]
Clearly $\SSS_{d,C} = \W_d$ for $d \in \N$.
To simplify the notation in proofs of results,
we write $\SSS_d$ and $T^\up_d$
instead of  $\SSS_{d,C}$ and $T^\up_{d,C}$, respectively, as long as
$C$ is considered to be fixed.

In the sequel we show that the sum operator is injective, and we
determine its range and
its inverse.
First of all, we note that
\begin{equation}\label{eq:C_OBdA_1}
T^\up_{d,C}(\SSS_{d,C}) = T^\up_{d,1}(\SSS_{d,1})
\end{equation}
for every $C > 0$.

\begin{lemma}\label{cmd:claim_iii_neu}
For every $C>0$ we have
$T^\up_{d,C}(\SSS_{d,C}) \subseteq \mN_d \subseteq \M_d$.
\end{lemma}

\begin{proof}
In view of \eqref{eq:C_OBdA_1}, we may assume that $C=1$. 
Let $\bg^\up := T^\up_d \bg$ for $\bg \in \SSS_d$. We show that
\begin{equation}\label{rep_dif_of_dif_gamma_up}
\left(\Delta_v \bg^\up \right)_u = 
\sum_{\substack{w\in \U_d \\ u\subseteq w, w\cap v = \emptyset}}
\gamma_w
\end{equation}
for all $u, v \in \U_d$.
Let us start with $|v|=0$, i.e., 
$v=\emptyset$. In that case we have
\[
\left(\Delta_v \bg^\up\right)_u = \bg^\up_u = 
\sum_{\substack{w\in \U_d \\ u\subseteq w, w\cap v = \emptyset}}
\gamma_w.
\]
Now let us assume that \eqref{rep_dif_of_dif_gamma_up} 
holds for $v\in \U_d$ and assume $s \in [d]\setminus v$. 
Due to Lemma~\ref{l2} we have
\begin{align*}
\left(\Delta_{v \cup \{s\}} \bg^\up\right)_u 
&=
\left(\Delta_{v} \bg^\up\right)_u -
\left(\Delta_{v} \bg^\up\right)_{u \cup \{s\}} \\
&= 
\sum_{\substack{w\in \U_d \\ u\subseteq w, w\cap v = \emptyset}}
\gamma_w
- \sum_{\substack{w\in \U_d \\ u\cup \{s\} \subseteq w, 
        w\cap v = \emptyset}}
\gamma_w \\
&= \sum_{\substack{w\in \U_d \\ u\subseteq w, 
         w\cap (v\cup \{s\}) = \emptyset}}
\gamma_w,
\end{align*}
which completes the proof of \eqref{rep_dif_of_dif_gamma_up}.
Since the right-hand side in \eqref{rep_dif_of_dif_gamma_up} is
non-negative, we obtain $\bg^\up \in \M_d$.

Furthermore, $\bg \in \SSS_d$ means that the family 
$(\gamma_v)_{v\in \U_d}$ is (unconditionally) summable. 
This implies, in particular, that 
$\gamma_u^{\up} = \sum_{\substack{v\in \U_d \\ u\subseteq v}}
\gamma_v$ tends to zero if 
$\max u$ tends to $\infty$. Hence $\bg^{\up} \in \mN_d$.
\end{proof}

\begin{rem}\label{r1}
Let $\bg \in \M_d$ and $u,v \in \U_d$ with $d=\infty$.
Due to Lemma~\ref{cmd:claim_i},
$\left(\Delta_{[s] \setminus v} \bg\right)_u$ 
with $s \in \N$ yields a non-increasing sequence in $[0,\gamma_u]$.
\end{rem}

We define a mapping 
\[
T^\lo_{d,C} \colon \M_d \to \W_d
\]
by
\[
\left(T^\lo_{d,C} \bg\right)_u :=
C^{-2|u|} \cdot
\begin{cases}
\left(\Delta_{[d] \setminus u} \bg\right)_u & \text{if $d < \infty$,}\\
\lim_{s \to \infty} \left(\Delta_{[s] \setminus u} \bg\right)_u & 
\text{if $d = \infty$.}
\end{cases}
\]
To simplify the notation in proofs of results,
we write $T^\lo_d$ instead of $T^\lo_{d,C}$, as long as
$C$ is considered to be fixed.

\begin{lemma}\label{cmd:claim_ii_neu}
For every $C>0$ and $\bg \in \M_d$ as well as
$u \in \U_d$ we have
\begin{equation}\label{g2}
\left(T^\lo_{d,C} \bg\right)_u
 = C^{-2|u|} \cdot
\sum_{\substack{v\in \U_d \\ u \subseteq v}} (-1)^{|v|-|u|} \gamma_v
\end{equation}
if $d \in \N$ and
\begin{equation}\label{g3}
\left(T^\lo_{d,C} \bg\right)_u
 = C^{-2|u|} \cdot \lim_{s \to \infty}
\sum_{\substack{v\in \U_s \\ u \subseteq v}} (-1)^{|v|-|u|} \gamma_v
\end{equation}
if $d = \infty$.
\end{lemma}

\begin{proof}
For $\bg \in \W_d$, $s \in [d]$, and $u \subseteq [s]$ we have
\begin{equation}\label{g7}
\left(\Delta_{[s]\setminus u} \bg\right)_u =
\sum_{w \subseteq [s] \setminus u} (-1)^{|w|} \gamma_{u \cup w} =
\sum_{\substack{v\in \U_s \\ u \subseteq v}} (-1)^{|v|-|u|} \gamma_v.
\end{equation}
For $d \in \N$ we take $s:=d$ to obtain \eqref{g2} from \eqref{g7}.
Now consider $\bg \in \M_d$ and $u \in \U_d$ with $d = \infty$. Here
we use \eqref{g7} with $s \geq \max u$ 
to obtain \eqref{g3}.
\end{proof}

Let us establish a useful identity.

\begin{lemma}\label{Lem:auxiliary_fact}
Let $\br := (\rho_u)_{u \in \U_d}\in \R^{\U_d}$. 
For $p,q \in [d]$ with $q \le p$
and $u \subseteq [q]$ we have
\[
\sum_{\substack{v\in \U_q \\ u \subseteq v}} 
\sum_{\substack{w\in \U_p \\ v \subseteq w}}
(-1)^{|v|} \rho_w 
=
(-1)^{|u|} 
\sum_{v \subseteq [p] \setminus [q]} \rho_{u \cup v}.
\]
\end{lemma}

\begin{proof}
We have 
\[
\sum_{\substack{v\in \U_q \\ u \subseteq v}} 
\sum_{\substack{w\in \U_p \\ v \subseteq w}}
(-1)^{|v|} \rho_w 
=
\sum_{\substack{w\in \U_p \\ u \subseteq w}}
\sum_{\substack{v\in \U_q \\ u \subseteq v \subseteq w}}
(-1)^{|v|} \rho_w 
=
\sum_{\substack{w\in \U_p \\ u \subseteq w}}
\rho_w
\sum_{\substack{v\in\U_q\\ u \subseteq v \subseteq w}} 
(-1)^{|v|}.
\]
For the innermost sum in the latter expression we obtain
\[
 \sum_{\substack{v\in\U_q\\ 
                 u \subseteq v \subseteq w }} 
(-1)^{|v|}
= (-1)^{|u|} \sum_{v' \subseteq (w \cap [q])\setminus u} (-1)^{|v'|}=
\begin{cases}
(-1)^{|u|} &\text{if $u = w \cap [q]$},\\
0 &\text{if $u \subsetneq w \cap [q]$.}
\end{cases}
\]
This establishes the claim.
\end{proof}

For the following result we refer to 
\citet[Lem.~1]{Matus94}, which deals with certain 
classes of convex cones in $\R^{\U_d}$, see also
\citet[4.18 and 4.19.(ii)]{Aigner79} for the unrestricted case
$\R^{\U_d}$.
For the convenience of the reader we present a proof of the result.

\begin{theo}\label{t1a}
Consider the case $d \in \N$. 
For every $C > 0$ the mapping $T^\up_{d,C}$ defines 
a bijection from $\SSS_{d,C}=\W_d$ onto 
\[
T^\up_{d,C}(\SSS_{d,C}) = \M_d
\]
with inverse mapping $T^\lo_{d,C}$.
\end{theo}

\begin{proof}
As it is easily verified,
it suffices to consider the case $C=1$.
According to Lemma~\ref{cmd:claim_iii_neu}
we have
$T^\up_{d}(\SSS_{d}) \subseteq \M_d$, and $\SSS_{d} = \W_d$
trivially holds true in the present case.

Let $\bg \in \SSS_d$ and $\bg^\up := T^\up_d \bg \in \M_d$.
We use \eqref{g2} and Lemma~\ref{Lem:auxiliary_fact} with 
$p := q := d$ to obtain 
\begin{align*}
\left(T^\lo_d \bg^\up\right)_u
&= 
(-1)^{|u|}
\sum_{\substack{v\in \U_d \\ u \subseteq v}} (-1)^{|v|} \bg^\up_v\\
&=
(-1)^{|u|}
\sum_{\substack{v\in \U_d \\ u \subseteq v}} 
\sum_{\substack{w\in \U_d \\ v \subseteq w}} 
(-1)^{|v|} 
\gamma_w\\
&=
\gamma_u
\end{align*}
for every $u \in \U_d$, i.e., $T^\lo_d \bg^\up = \bg$. In
particular, $T^\up_d$ is an injection.

Let $\bg \in \M_d$ and $\bg^\lo := T^\lo_d \bg$.
Again, we use \eqref{g2} and
Lemma~\ref{Lem:auxiliary_fact} with $p := q := d$ 
to obtain
\[
\left(T^\up_d \bg^\lo\right)_u 
= 
\sum_{\substack{v\in \U_d \\ u \subseteq v}} 
\gamma^{\lo}_v
=  
\sum_{\substack{v\in \U_d \\ u \subseteq v}}
\sum_{\substack{w\in \U_d \\ v \subseteq w}} (-1)^{|w|-|v|} \gamma_w
=
\gamma_u
\]
for every $u \in \U_d$, i.e., $T^\up_d \bg^\lo = \bg$.
In particular, $\M_d \subseteq T^\up_d (\SSS_d)$.
\end{proof}

Next, we consider the case $d = \infty$.
In addition to $\mN_d$ we consider the subsets
\[
\PP_d :=
\left\{ \bg \in \M_d : \sum_{v \in \U_d} \gamma_v < \infty \right\}
\]
and 
\begin{equation}\label{def_A_d}
\A_d:= 
\left\{\bg \in \M_d : 
\lim_{r \to \infty} \lim_{s \to \infty} 
\Delta_{[s] \setminus [r]} \bg = \bg
\right\}
\end{equation}
of $\M_d$.

\begin{theo}\label{t1b}
Consider the case $d = \infty$. For every $C > 0$ we have
\begin{equation}\label{range_T_up}
\PP_d \subsetneq \A_d = T^\up_{d,C} (\SSS_{d,C})
\subsetneq \mN_d \subsetneq \M_d,
\end{equation}
and the mapping $T^\up_{d,C}$ defines a bijection from $\SSS_{d,C}$ onto 
$\A_d$ with inverse mapping $T^\lo_{d,C}|_{\A_d}$. Furthermore,
\begin{equation}\label{range_T_lo}
T^\lo_{d,C}(\M_d)  = \SSS_{d,C},
\end{equation}
and
\begin{equation}\label{g77}
(T^{\up}_{d,C} \circ T^{\lo}_{d,C}) \bg \leq \bg 
\end{equation}
for every $\bg \in \M_d$ with equality if and only if $\bg \in \A_d$.
\end{theo}

The major findings from Theorems~\ref{t1a} and \ref{t1b} may be
reformulated in terms of finite measures on the power set of
$\U_d$, see Appendix~\ref{r8}. 
In this way we realize a particular instance of $\sigma$-continuity
in the definition of $\A_d$. 

For the proof of Theorem~\ref{t1b}, which is split into
several lemmata, it suffices to consider the case $C=1$.

\begin{lemma}\label{l31}
For $d = \infty$ and every $\bg \in \M_d$
we have $T^\lo_d \bg \in \SSS_d$ and
\[
\left(T^\up_d \circ T^\lo_d\right) \bg 
=
\lim_{r \to \infty} 
\lim_{s \to \infty} \left( \Delta_{[s] \setminus [r]} \bg \right)
\leq \bg.
\]
\end{lemma}

\begin{proof}
Let $\bg \in \M_d$ and $\bg^\lo := T^\lo_d \bg$.
Consider $u\in \U_{d}$ and let $r \in [d]$ with 
$u \subseteq [r]$.
Then
\[
A(r,u) :=
\sum_{\substack{v\in \U_r \\ u \subseteq v}} \gamma^{\lo}_v
=   \lim_{s\to \infty} \sum_{\substack{v\in \U_r \\ u \subseteq v}}
\sum_{\substack{w\in \U_s \\ v \subseteq w}} (-1)^{|w|-|v|} \gamma_w,
\]
see \eqref{g3}.
We use 
Lemma~\ref{Lem:auxiliary_fact} with $p := s$ and $q := r$
together with Remark~\ref{r1} to obtain
\[
A(r,u)
= \lim_{s\to \infty} \sum_{v \subseteq [s] \setminus [r]}
(-1)^{|v|} \gamma_{u \cup v}
=
\lim_{s \to \infty} \left( \Delta_{[s] \setminus [r]} \bg \right)_u 
\in [0,\gamma_u].
\]
Furthermore, due to Lemma~\ref{cmd:claim_i},
$A(r,u) \leq A(r+1,u)$. Consequently, 
\[
\sum_{\substack{v\in \U_d \\ u \subseteq v}} 
\gamma^{\lo}_v
= \lim_{r \to \infty} A(r,u) \leq \gamma_u.
\]
Take $u := \emptyset$ to conclude that 
$\bg^\lo \in \SSS_d$.
By definition, the left-hand side in the last estimate is
equal to $(T^\up_d \bg^\lo)_u$.
\end{proof}

According to Lemma~\ref{l31} we have \eqref{g77} and obtain for 
the set $\A_d$ of weights the representation
\begin{equation}\label{eq:repAd}
\A_d = \left\{ \bg \in \M_d : 
(T^{\up}_{d} \circ T^{\lo}_{d}) \bg = \bg \right\}.
\end{equation}

\begin{lemma}\label{l30}
For $d = \infty$ and every $\bg \in \SSS_d$ 
we have
\[
(T^\lo_d \circ T^\up_d) \bg = \bg.
\]
\end{lemma}

\begin{proof}
Let $\bg \in \SSS_d$ and $\bg^\up := T^\up_d \bg \in \M_d$, see
Lemma~\ref{cmd:claim_iii_neu}.
For $r,s \in [d]$ with $s \leq r$ and $u \in \U_d$ we put
\[
a(r,s,u) := 
\sum_{v \subseteq [r] \setminus [s]} 
\gamma_{u \cup v}.
\]
Let us first consider the case $u \subseteq [s]$. Then
the sequence $(a(r,s,u))_{r \in [d]}$ 
converges since $\bg \in \SSS_d$, and
Lemma~\ref{Lem:auxiliary_fact} with $p := r$ and $q :=s$ yields 
\[
\sum_{\substack{v\in \U_s \\ u \subseteq v}} (-1)^{|v|} \bg^\up_v
=
\lim_{r \to \infty}
\sum_{\substack{v\in \U_s \\ u \subseteq v}} 
\sum_{\substack{w\in \U_r \\ v \subseteq w}} 
(-1)^{|v|} 
\gamma_w
= (-1)^{|u|} \lim_{r \to \infty} a(r,s,u).
\]
If we now consider arbitrary $u \in \U_d$ and let $s$ tend to 
infinity, 
we clearly get $u \subseteq [s]$ for sufficiently large $s$. Hence we may
use \eqref{g3} to obtain, for every $u \in \U_d$,
\begin{equation}\label{g66}
\left(T^\lo_d \bg^\up\right)_u
= 
(-1)^{|u|} \lim_{s \to \infty}
\sum_{\substack{v\in \U_s \\ u \subseteq v}} (-1)^{|v|} \bg^\up_v
=
\lim_{s \to \infty} \lim_{r \to \infty} a(r,s,u).
\end{equation}
Moreover, if $u \subseteq [s]$,
\[
0 \leq a(r,s,u) - 
\gamma_u = 
\sum_{\emptyset \neq v \subseteq [r] \setminus [s]}
\gamma_{u \cup v} \leq b(s,u)
\]
with
\[
b(s,u) :=
\sum_{\substack{{v \in \U_d} \\ v \neq \emptyset, v \cap [s] = 
\emptyset}
}\gamma_{u \cup v},
\]
and therefore
\[
0 \leq \lim_{r \to \infty} a(r,s,u) - 
\gamma_u \leq b(s,u).
\]
For arbitrary $u \in \U_d$ we proceed as
follows. Since $\bg \in \SSS_d$ implies
\[
\sum_{\substack{v \in \U_d \\ v \cap u = \emptyset}}
\gamma_{u \cup v} < \infty,
\]
we conclude that $\lim_{s \to \infty} b(s,u) = 0$, and therefore
\begin{equation}\label{g67}
\lim_{s \to \infty} \lim_{r \to \infty} a(r,s,u) = 
\gamma_u.
\end{equation}
Combining \eqref{g66} and \eqref{g67} we obtain
$T^\lo_d \bg^\up =\bg$, as claimed.
\end{proof}

Lemma~\ref{l30} yields
the injectivity of $T^\up_d$.
We conclude that the mappings
$T^\up_d \colon \SSS_d \to T^\up_d(\SSS_d)$ and 
$T^\lo_d|_{T^\up_d(\SSS_d)} \colon T^\up_d(\SSS_d) \to \SSS_d$ are 
bijections and 
$T^\lo_d|_{T^\up_d(\SSS_d)}$ is the inverse of $T^\up_d$. 
In combination with \eqref{eq:repAd} this implies 
$T^\up_d(\SSS_d) \subseteq \A_d$. Moreover, we have
\eqref{range_T_lo}, since 
$T^\lo_d (\M_d) \subseteq \SSS_d$ by Lemma~\ref{l31}
and
$\SSS_d \subseteq T^\lo_d (\M_d)$ by Lemma~\ref{l30}.
For later reference we explicitly note that
\begin{equation}\label{g79}
T^\lo_d \left(T^\up_d(\SSS_d)\right) = \SSS_d.
\end{equation}

\begin{lemma}\label{l32}
For $d = \infty$ 
we have $\A_d \subseteq T^\up_d(\SSS_d)$.
\end{lemma}

\begin{proof}
Let $\bg \in \A_d$.
Due to \eqref{range_T_lo} and \eqref{g79}, there exist weights 
$\widehat{\bg} \in T^\up_d(\SSS_d)$ such that 
$T^\lo_{d} \widehat{\bg} = T^\lo_d \bg$.  It follows
from \eqref{eq:repAd} that
\[
\widehat{\bg} = 
\left( T^\up_d \circ T^\lo_d \right) \widehat{\bg}
=
\left( T^\up_d \circ T^\lo_d \right) \bg = \bg,
\]
so that $\bg \in T^\up_d(\SSS_d)$.
\end{proof}

Lemma~\ref{l32} completes the proof of 
$\A_d = T^\up_d(\SSS_d)$.
It remains to establish the proper inclusions that are stated
in \eqref{range_T_up}.

\begin{lemma}\label{l35}
For $d = \infty$ we have
$T^\up_{d} (\SSS_{d}) \subsetneq \mN_d \subsetneq \M_d$.
\end{lemma}

\begin{proof}
According to Lemma~\ref{cmd:claim_iii_neu}, we have
$T^\up_{d} (\SSS_{d}) \subseteq \mN_d \subseteq \M_d$.

Let $\bg \in \W_d$ be constantly equal to $c > 0$. 
Since $\Delta_v \bg = 0$ for every 
non-empty set $v \in \U_d$, we have $\bg \in \M_d$,
but obviously $\bg \notin \mN_d$.

Let $\bg \in \W_d$ be given by $\gamma_u := \prod_{j \in u} 1/j$
for $u \in \U_d$.
Due to Proposition~\ref{l5} below we have 
$\bg \in \mN_d$, but $\bg \notin T^\up_d (\SSS_d)$.
\end{proof}

\begin{lemma}\label{l36}
For $d = \infty$ we have
$\PP_d \subsetneq T^\up_{d} (\SSS_{d})$.
\end{lemma}

\begin{proof}
Let $\bg \in \M_d$.
We use Lemma~\ref{cmd:claim_i} to obtain
\begin{equation*}
0 \leq 
\left( \Delta_\emptyset \bg \right)_u -
\left( \Delta_{[s] \setminus [r]} \bg \right)_u 
= \gamma_u -
\left( \Delta_{[s] \setminus [r]} \bg \right)_u 
\leq \sum_{\substack{v \in \U_d \\  
v \neq \emptyset,v \cap [r] = \emptyset}} \gamma_{u\cup v}.
\end{equation*}
For $\bg \in \PP_d$
the right-hand side in the last estimate tends to zero for 
$r$ tending to $\infty$,
so that $\bg \in \A_d = T^\up_d(\SSS_d)$.

Let $\bg \in \SSS_d$, i.e.,
\[
\sum_{v\in \U_{d}} 
\gamma_v < \infty.
\]
Note that the weights $T^\up_d \bg$ are summable 
if and only if 
\begin{equation*}
\sum_{v\in \U_{d}} 2^{\left|v\right|}\gamma_v < \infty.
\end{equation*}
For $\bg \in \W_d$ given by
$\gamma_u := s^{-2}$ for 
$u := \{1,\dots,s\}$ with $s \in \N$ and $\gamma_u := 0$
otherwise we thus obtain $\bg \in \SSS_d$ and $T^\up_d \bg \not\in \PP_d$.
\end{proof}

This completes the proof of Theorem~\ref{t1b}.

We continue to study the case $d=\infty$, and we put 
$$
\bg^{1/\tau} := (\gamma_u^{1/\tau})_{u \in \U_d}
$$ 
for $\bg \in \W_d$ and $\tau > 0$. 
The following lemma allows us to
relate the decay of weights 
$\bg \in \SSS_{d,C}$ to the decay of 
the transformed weights
$T^\up_{d,C} \bg$, see Section~\ref{Subsec:Decay}.

\begin{lemma}\label{l3}
Assume that $d=\infty$. Let $\bg \in \SSS_{d,C}$ with 
$C>0$. Put $\bg^\up := T^\up_{d,C} \bg$. We have
\[
\sum_{u \in \U_d} \left(\gamma_u^\up\right)^{1/\tau} < \infty 
\quad \Rightarrow \quad
\bg^{1/\tau} \in \SSS_{d,C^{1/\tau}}
\]
for every $\tau > 0$ and
\[
\bg^{1/\tau} \in \SSS_{d, \sqrt{2} C^{1/\tau}}
\quad \Rightarrow \quad
\sum_{u \in \U_d} \left(\gamma_u^\up\right)^{1/\tau} < \infty 
\]
for every $\tau \geq 1$.
\end{lemma}

\begin{proof}
Let $\tau > 0$.
We use $\gamma_u^\up \ge C^{2|u|} \, \gamma_u$ to obtain
\[
\sum_{u\in \U_d} \left(\gamma^\up_u\right)^{1/\tau} 
\geq
\sum_{u\in\U_d} C^{2|u|/\tau}\, \gamma_u^{1/\tau},
\]
which yields the first implication. 

For the proof of the second implication in the case $\tau \ge 1$ 
we employ
an inequality sometimes referred to as Jensen's inequality. This inequality
states that $(\sum y_i)^p \le \sum y_i^p$, if we consider non-negative 
$y_i$ and $p\in {(0,1]}$. Therefore
\begin{align*}
\sum_{u\in \U_d} \left(\gamma^\up_u\right)^{1/\tau}
&= \sum_{u\in \U_d} 
\left(\sum_{\substack{v\in\U_d \\ u\subseteq v}} 
C^{2|v|}\, \gamma_v\right)^{1/\tau}
\le \sum_{u\in \U_d} 
\sum_{\substack{v\in\U_d \\ u\subseteq v}} 
C^{2|v|/\tau}\, \gamma_v^{1/\tau}\\
&= \sum_{v\in\U_d} C^{2|v|/\tau}\, \gamma_v^{1/\tau} 
\sum_{\substack{u\in\U_d \\ u\subseteq v}} 1
= \sum_{v\in\U_d} 2^{|v|}\, C^{2|v|/\tau}\, \gamma_v^{1/\tau}
\end{align*}
for every $\tau \geq 1$, which yields the second implication.
\end{proof}

\section{Particular Classes of Weights}
\label{SEC:Part_Weights}

In the sequel, we study some important classes of weights of
a particular structure, namely:
\begin{itemize}
\item[(i)]
\emph{Product weights}, which are of the form
\[
\gamma_u = \prod_{j\in u} \gamma_j
\]
for all $u\in \U_d$, where $(\gamma_j)_{j\in [d]}$ 
is any non-increasing sequence
of non-negative real numbers.
By convention the empty product is 
equal to one, i.e., $\gamma_{\emptyset} = 1$. 
\item[(ii)]
\emph{Product and order-dependent (POD) weights} 
of the form
\[
\gamma_u = \Gamma_{|u|} \prod_{j\in u} \gamma_j
\]
for all $u \in \U_d$ with $(\gamma_j)_{j \in [d]}$ as in (i) 
and with a sequence $(\Gamma_k)_{k \in [d] \cup \{0\}}$ of
non-negative real numbers that satisfies
for some constants $a, C_a >0$ the estimate
\[
\Gamma_k \le C_a (k!)^a 
\hspace{3ex}\text{for all $k\in [d] \cup \{0\}$.}
\]
Note that the latter property trivially holds true 
for $d \in
\N$. For POD weights as before with $d = \infty$ we put
\[
p := \decay((\gamma_j)_{j \in \N}). 
\]
\item[(iii)]
\emph{Finite-order weights of order $\omega \in \N$}, which satisfy
\begin{equation*}
\gamma_u = 0
\hspace{2ex}\text{for all $u\in\U_d$ with $|u|>\omega$.}
\end{equation*}
\end{itemize}

Weighted function spaces, and in particular, spaces
involving product weights have been introduced in \citet{SW98}.
Product and order-dependent weights are introduced in \citet{KSS12}
as a generalization of product weights. 
Finite-order weights are introduced for spaces of functions with a
finite number of variables in \citet{DSWW06}. Subclasses include, 
e.g., finite-diameter weights, see \citet{NW08}, finite-intersection 
weights, see \citet{Gne10}, and finite-projective dimension (FPD) 
weights, see \citet{DG12}.

\subsection{Complete Monotonicity}

At first,
we discuss whether $\bg \in \M_d$ 
for any $\bg$ of the form (i), (ii), or (iii).

\begin{rem}
Due to Lemma~\ref{cmd:claim_i}, $\bg \in \M_d$ implies, in particular, 
$\gamma_v \geq \gamma_w$ for all $v,w\in \U_d$ with $v \subset w$.
Obviously, this monotonicity property need not hold for 
finite-order weights. 

The same holds true for POD weights of the form (ii),
which has already been pointed out in \citet{DKS13}.
Indeed, assume that the POD weights are of the particular form
\begin{equation}\label{pod_weights_form_DKS}
\gamma_u = \left((|u|+\ell)!\right)^a  \prod_{j\in u} \gamma_j
\end{equation}
for some $\ell \in \N$ and for all $u \in \U_d$,
as considered in \citet[p.~208]{DKS13}.
Assume that $\gamma_j >0$ for all $j\in[d]$. Then
for $v := \{2,\dots,s\}$ and $w := \{1,\dots,s\}$
with $s \in [d]$  we have equivalence of 
$\gamma_v \geq \gamma_w$ and
\[
\gamma_1 \leq \frac{1}{\left(s+\ell \right)^a}.
\]
For $d \in \N$ this implies 
\[
\gamma_i \leq \frac{1}{\left(d+\ell \right)^a}
\]
for every $i \in [d]$, so that for POD weights of the 
form~\eqref{pod_weights_form_DKS} with $\bg \in \M_d$ 
and strictly positive $\gamma_i$, $i=1,\ldots, [d]$, the weights 
$\gamma_i$ have to be rather small. For $d = \infty$
POD weights of the form \eqref{pod_weights_form_DKS} with all 
$\gamma_j >0$ cannot be monotonically
decreasing and thus, in particular, 
do not satisfy $\bg\in\M_d$.   
\end{rem}

For product weights the property $\bg \in \M_d$,
as well as $\bg \in \mN_d$, $\A_d$, or $\PP_d$ in the case
$d=\infty$,
is easily characterized as follows.

\begin{prop}\label{l5}
Let $\bg$ be product weights according to (i). We have 
\begin{equation}\label{eq:diff_prod_weights}
\left(\Delta_v \bg\right)_u 
= \gamma_{u} \prod_{j\in v} (1 - \gamma_{j})
\end{equation}
for $u, v\in \U_d$ with $u\cap v = \emptyset$.
Moreover,
\begin{equation}\label{char_prod_weights_M_d}
\bg \in \M_d \quad \Leftrightarrow \quad 
\forall\, j \in [d] \colon 0 \leq \gamma_j \leq 1.
\end{equation}
If additionally $d=\infty$ and $\bg \in \M_d$, then we have the 
following equivalences:
\begin{itemize}
\item[{\rm (}a{\rm )}] 
$\bg \in \mN_d \Leftrightarrow \lim_{j\to \infty} \gamma_j = 0$,
\item[{\rm (}b{\rm )}] 
$T^\lo_d \bg = 0 \Leftrightarrow \sum_{j=1}^\infty \gamma_j = \infty$,
\item[{\rm (}c{\rm )}] 
$\bg \in \A_d \Leftrightarrow 
\sum_{j=1}^\infty \gamma_j <  \infty
\Leftrightarrow \bg \in \PP_d$. 
\end{itemize}
\end{prop}

\begin{proof}
For product weights $\bg$ according to (i) we have
\[
\left(\Delta_v \bg\right)_u 
= \sum_{w\subseteq v} (-1)^{|w|} \gamma_{u\cup w}
= \gamma_{u} \sum_{w\subseteq v} (-1)^{|w|} \prod_{j\in w} \gamma_{ j}
= \gamma_{u} \prod_{j\in v} (1 - \gamma_{j})
\]
for $u, v\in \U_d$ with $u\cap v = \emptyset$, which yields 
\eqref{eq:diff_prod_weights} and, due to the definition of $\M_d$, 
also \eqref{char_prod_weights_M_d}.

Now let $d=\infty$ and $\bg \in \M_d$. Concerning (a): Note that due 
to \eqref{char_prod_weights_M_d} we have $\gamma_u \le \gamma_{\max u}$ 
for all $\emptyset \neq u\in \U_d$. Hence 
$\lim_{j \to \infty} \gamma_j = 0$ implies 
$\lim_u \gamma_u = 0$, i.e., $\bg \in \mN_d$.  
The reverse implication follows trivially from 
$\gamma_j = \gamma_{\{j\}}$ for all $j\in\N$.

Concerning (b): Due to
\eqref{eq:diff_prod_weights} we have
\begin{equation*}
C^{2|u|} \left( T^\lo_d \bg \right)_u = 
\lim_{s\to \infty} \left( \Delta_{[s] \setminus u} \bg \right)_u
= 
\gamma_u \lim_{s\to \infty} \prod_{j\in [s]\setminus u} (1-\gamma_j).
\end{equation*}

Let $u_1 := \{ j \in \N : \gamma_j = 1\}$.
If $u_1$ is an infinite set (which is equivalent to $u_1 = \N$), then
$\sum_{j\in\N} \gamma_j = \infty$ and $T^\lo_d \bg=0$.

So let us assume that $u_1$ is finite.
Suppose first that $T^\lo_d \bg=0$. Then we choose $u=u_1$ and obtain
\begin{equation}\label{eq:lim_prod_0_u1}
\lim_{s\to \infty} \prod_{j\in [s]\setminus u_1} (1- \gamma_j) = 0,
\end{equation}
as $\gamma_{u_1}\neq 0$. Moreover, 
\eqref{eq:lim_prod_0_u1} is equivalent to
\[
 \lim_{s\to \infty} \sum_{j\in [s]\setminus u_1} \gamma_j = \infty,
\]
which, in turn, is equivalent to $\sum_{j\in\N} \gamma_j =\infty$. 

Suppose on the other hand that $\bg$ is such that 
$\sum_{j\in\N} \gamma_j =\infty$. Then we have, for all $v\in \U_d$,
\[
 0 \le \prod_{j\in v} (1-\gamma_j) \le 
\exp \left(-\sum_{j\in v}\gamma_j\right), 
\]
and thus, for all $u\in\U_d$, 
\[
 \lim_{s\to \infty} \prod_{j\in [s]\setminus u} (1- \gamma_j) = 0.
\]
This completes the proof of (b).

Concerning (c): Recall $\bg \neq 0$. If 
$\bg \in \A_d$, then Theorem~\ref{t1b}
assures that $\left(T^\up_d \circ T^\lo_d \right) \bg = \bg$. 
Hence $T^\lo_d \bg \neq 0$ and (b) implies 
$\sum_{j\in\N} \gamma_j < \infty$. In case we have 
$\sum_{j\in\N} \gamma_j < \infty$, we obtain 
$\sum_{u\in \U_d} \gamma_u  = 
\prod_{j\in\N} (1 + \gamma_j) < \infty$, 
i.e., $\bg \in \PP_d$. Finally, $\PP_d \subseteq
\A_d$, see Theorem~\ref{t1b}.
\end{proof}

\subsection{Summability}

Next, we discuss whether $\bg \in \SSS_{d,C}$ for any 
$\bg$ of the form (i), (ii), or (iii) in the non-trivial 
case $d=\infty$.

To do so for POD weights, we need 
the following result,
which is essentially a variation of the statement in 
\citet[Thm. 5]{DG12}.

\begin{lemma}\label{prel4}
Assume that $d = \infty$. Let $C>0$ and let $\bg$ be
POD weights with $a$ and $p$ according to (ii).
Then we have
\begin{equation}\label{iq:trivial_lower_bound}
\sum_{\substack{u \subseteq [j] \\ j\in u}}
C^{2|u|} \gamma_u \ge C^2  \Gamma_1 \gamma_j
\end{equation}
for all $j\in\N$. Furthermore, 
if either 
\begin{equation}\label{cond_p>a}
p>a \hspace{3ex}\text{and}\hspace{3ex} \sum_{j\in\N}\gamma_j < \infty
\end{equation}
or 
\begin{equation}\label{cond_p=a}
p=a \ge 1 \hspace{3ex}\text{and}\hspace{3ex}
\sum_{j\in\N} \left( C^{2} \gamma_j \right)^{1/p} < 1
\end{equation}
holds,  then there exists a constant $c >0$
such that 
\begin{equation}\label{dg_14}
\sum_{\substack{u \subseteq [j] \\ j\in u}}
C^{2|u|} \gamma_u \le c \gamma_j
\end{equation}
for all $j\in\N$, where $c$ may depend on $C$ and $\bg$.
\end{lemma}

\begin{proof}
Since $\gamma_{\{j\}} = \Gamma_1 \gamma_j$, estimate 
\eqref{iq:trivial_lower_bound} obviously holds true.

To prove the second statement of the lemma, we use arguments very similar 
to the ones used in \citet{DG12} in the proof of Theorem~5.
We first prove estimate~\eqref{dg_14}
under the assumption that 
\begin{equation*}
p \ge \max (1,a) 
\hspace{3ex}\text{and}\hspace{3ex}
\sum_{j\in\N} \left( C^{2} \gamma_j \right)^{1/p} < 1. 
\end{equation*}
Since $p \ge \max(1, a)$, we may use
Jensen's inequality to obtain
\begin{align*}
\sum_{\substack{u \subseteq [j] \\ j\in u}}
C^{2|u|} \gamma_u &= 
\sum_{\substack{u \subseteq [j] \\ j\in u}}
\Gamma_{|u|} 
\prod_{i\in u} \left( C^2 \gamma_i \right) \\
&\le  
\sum_{\substack{u \subseteq [j] \\ j\in u}}
C_a \left( |u|! \right)^a \prod_{i \in u} \left( C^{2} \gamma_i \right) \\
&\le C_a \left( 
\sum_{\substack{u \subseteq [j] \\ j\in u}}
|u|!  \prod_{i \in u} \left( C^{2} \gamma_i \right)^{1/p} \right)^p.
\end{align*}
We now use for a multi-index 
$\nu =(\nu_1,\ldots,\nu_j) \in \N_0^j$ the notations 
$|\nu| := \nu_1 + \cdots + \nu_j$ and 
$\nu! := \prod_{i=1}^j \nu_i!$. 
Moreover, we use the shorthand
\[
T_p(C) :=  \sum_{i \in \N}  \left( C^{2} \gamma_i \right)^{1/p}.
\]
Then we obtain with the help of the 
multinomial formula and the formula for the (convergent) geometric series,
\begin{align}\label{assumptions_proof}
&
\sum_{\substack{u \subseteq [j] \\ j\in u}}
|u|!  \prod_{i \in u} \left( C^{2} \gamma_i \right)^{1/p}  
= \sum_{\substack{\nu \in \{0,1\}^j \\ \nu_j \neq 0}} 
\frac{|\nu|!}{\nu!}  \prod_{i = 1}^j 
\left( C^{2} \gamma_i \right)^{\nu_i/p} \notag \\
&\qquad \le \sum_{\substack{\nu \in \N_0^j \\ \nu_j \neq 0}} 
\frac{|\nu|!}{\nu!}  \prod_{i = 1}^j 
\left( C^{2} \gamma_i \right)^{\nu_i/p} \notag \\
&\qquad = \sum_{\kappa \in \N} \left(  
\sum_{\substack{\nu \in \N_0^j \\ |\nu| = \kappa}} \frac{\kappa!}{\nu!}  
\prod_{i = 1}^j \left( C^{2} \gamma_i \right)^{\nu_i/p} 
- \sum_{\substack{\nu \in \N_0^{j-1} \\ |\nu| = \kappa}} 
\frac{\kappa!}{\nu!}  
\prod_{i = 1}^{j-1} \left( C^{2} \gamma_i \right)^{\nu_i/p} \right)
\notag \\
&\qquad = \sum_{\kappa \in \N} \left(  
\left( \sum_{i=1}^j  \left( C^{2} \gamma_i \right)^{1/p}  \right)^\kappa 
-   
\left( \sum_{i=1}^{j-1}  \left( C^{2} \gamma_i \right)^{1/p}
 \right)^\kappa \, \right) \notag \\
&\qquad = \frac{1}{1 - \sum_{i=1}^j  
\left( C^{2} \gamma_i \right)^{1/p} }
- 
\frac{1}{1 - \sum_{i=1}^{j-1}  
\left( C^{2} \gamma_i \right)^{1/p} } \notag \\
&\qquad  \le \frac{ \left( C^{2} \gamma_j \right)^{1/p}}
{\left(1 - 
T_p(C)
\right)^2}.
\end{align}
Altogether, we get 
\begin{equation*}
\sum_{\substack{u \subseteq [j] \\ j\in u}}
C^{2|u|} \gamma_u 
\le C_a \left( 1-T_p(C) \right)^{-2p} C^2 
\gamma_j. 
\end{equation*}
In particular, we have shown
that condition~\eqref{cond_p=a} implies inequality~\eqref{dg_14}. 

Now assume that condition~\eqref{cond_p>a} holds.
We consider the non-trivial case $\gamma_1 > 0$, and we put
$\tau := 1$ if $a < 1$ and choose $\tau \in (a,p)$
if $a \geq 1$. 

We put $R := \left( 2  
\sum_{i \in \N} \left( C^{2}\gamma_i \right)^{1/\tau} \right)^{\! \tau}$ 
and note that, by the definition of $p$ in (ii),
we have $R\in (0,\infty)$. Moreover, we define
$\widetilde{C} := C/\sqrt{R}$, and we note that 
\begin{equation}\label{eq:technical}
(|u|!)^\alpha \cdot x^{|u|} \leq
\exp\left(x^{1/(1-\alpha)}\right) \cdot |u|!
\end{equation}
for every $x > 0$ and every $\alpha \in (0,1)$.

Due to Jensen's inequality we obtain
\begin{align*}
\sum_{\substack{u \subseteq [j] \\ j\in u}}
C^{2|u|} \gamma_u 
&\le   
\sum_{\substack{u \subseteq [j] \\ j\in u}}
C_a \left( |u|! \right)^a \prod_{i \in u} \left( C^{2} \gamma_i \right) \\
&=   C_a 
\sum_{\substack{u \subseteq [j] \\ j\in u}}
\left( |u|! \right)^a  R^{|u|} \prod_{i \in u} \left(
\widetilde{C}^{2} \gamma_i \right) \\
& \le	 C_a \left(
\sum_{\substack{u \subseteq [j] \\ j\in u}}
\left( |u|! \right)^{a/\tau}
R^{|u|/\tau} \prod_{i \in u} \left(
\widetilde{C}^{2} \gamma_i \right)^{1/\tau}
\right)^{\tau}\\
&\le  C_a 
\exp(\tau R^{1/(\tau-a)})
\left(
\sum_{\substack{u \subseteq [j] \\ j\in u}}
|u|!   \prod_{i \in u} \left( \widetilde{C}^{2} \gamma_i
\right)^{1/\tau} \right)^{\tau} \\
&\le 
C_a  \exp(\tau R^{1/(\tau-a)})
\left( 1-T_{\tau} (\widetilde{C}) \right)^{-2\tau} \widetilde{C}^2 
\gamma_j,
\end{align*}
where in the penultimate step we have used 
\eqref{eq:technical} with $\alpha=a/\tau$ 
and $x=R^{1/\tau}$, and in the last step we have used $\sum_{i \in \N} 
\left(\widetilde{C}^2 \gamma_i\right)^{1/\tau} = 1/2 < 1$ and 
\eqref{assumptions_proof} 
(with $p$ replaced by $\tau$ and $C$ replaced by $\widetilde{C}$).
\end{proof}

The following result is well known for product weights,
and for variants in the case of POD weights we refer to
\citet[Lem.~6.2]{KSS12} and \citet[Sec.~4.3]{DG12}.

\begin{prop}\label{l4}
Assume that $d = \infty$.

For product weights $\bg$ according to (i) and every $C > 0$ 
we have
\[
\bg \in \SSS_{d,C} \quad \Leftrightarrow \quad 
\sum_{j \in \N} \gamma_j < \infty.
\]

For POD weights $\bg$ with $a$ and $p$ according to (ii) 
and for every $C>0$ the following statement holds:
If $p>a$ and $\Gamma_1>0$, then 
\[
\bg \in \SSS_{d,C} \quad \Leftrightarrow \quad 
\sum_{j \in \N} \gamma_j < \infty.
\]
Furthermore, if $p=a \ge 1$, then 
$\sum_{j\in \N} \left( C^2 \gamma_j \right)^{1/p} 
< 1$ implies $\bg \in \SSS_{d,C}$.

For finite-order weights $\bg$ 
according to (iii) and every $C>0$ we have 
\[
\bg \in \SSS_{d,C} \quad \Leftrightarrow \quad \sum_{u \in \U_d}
\gamma_u < \infty.
\]
\end{prop}

\begin{proof}
In the case (i) the statement follows from
\[
\sum_{u \in \U_d} C^{2|u|} \gamma_u = \prod_{j \in \N} (1+ C^2 \gamma_j).
\]

Consider the case (ii). Assume first that $\bg \in \SSS_{d}$. 
If $\Gamma_1 >0$, then estimate~\eqref{iq:trivial_lower_bound} 
implies $\sum_{j\in \N} \gamma_j < \infty$. 
Assume now that either $p>a$ and $\sum_{j\in \N} \gamma_j < \infty$
or that $p=a \ge 1$ and $\sum_{j\in \N} \left( C^2 \gamma_j \right)^{1/p} 
< 1$. 
Due to Lemma~\ref{prel4} we have
estimate~\eqref{dg_14}, which yields $\bg \in \SSS_d$.

For finite-order weights $\bg$ of order $\omega \in \N$ we have 
\begin{equation*}
\min\{1,C^{2\omega}\} \sum_{u\in\U_d}\gamma_u \le  
\sum_{u\in\U_d} C^{2|u|} \gamma_u  
\le \max\{1,C^{2\omega}\} \sum_{u\in\U_d}\gamma_u,
\end{equation*} 
hence $\bg \in \SSS_d$ is satisfied if and  
only if the weights $(\gamma_u)_{u\in \U_d}$ are summable. 
\end{proof}

\subsection{Invariance Properties}

Now,
we consider the following question:
If $\bg$ belongs to one of the classes specified by (i), (ii), or
(iii), does $T^\up_{d,C} \bg$ or $T^\lo_{d,C} \bg$ 
belong to the same class?

\begin{rem}\label{r3}
Let $\bg$ be finite-order weights of order $\omega$. From the definition
of $T^\up_{d,C}$ and the representation of
$T^\lo_{d,C}$ according
to Lemma~\ref{cmd:claim_ii_neu} we see immediately that the
weights $T^{\up}_{d,C} \bg$ and $T^{\lo}_{d,C} \bg$ 
are again finite-order 
weights of the same order $\omega$ (as long as $\bg$ lies in
$\SSS_{d,C}$ or $\M_d$, respectively).
\end{rem}

\begin{rem}\label{r2}
Let $\bg$ be product weights according to (i). 
If $\bg \in \SSS_{d,C}$, then we have
\[
\Gamma_\emptyset := \prod_{j \in [d]} (1 + C^2 \gamma_j) < \infty,
\] 
which follows in the case $d=\infty$ from Proposition~\ref{l4}. 
Furthermore, for $u \in \U$,
\begin{align*}
\left(T^\up_{d,C} \bg\right)_u &= 
C^{2|u|} \gamma_u 
\sum_{\substack{v \in \U_d \\ v \cap u = \emptyset}} C^{2|v|} \gamma_v 
= 
C^{2|u|} \prod_{j \in u} \gamma_j 
\prod_{j \in [d] \setminus u} (1 + C^2 \gamma_j) \\
&=
\Gamma_\emptyset 
\prod_{j \in u} \frac{C^2 \gamma_j}{1+C^2\gamma_j}.
\end{align*}
We conclude that $T^\up_{d,C} \bg$ are product weights if
and only if $\Gamma_\emptyset = 1$, which is equivalent to
the trivial case $\gamma_j = 0$ for all $j\in[d]$. In any case, 
$T^\up_{d,C} \bg$ are very particular POD weights, namely
with $\Gamma_k := \Gamma_\emptyset$ for $k \in [d]$.

The natural question to ask next is whether the image 
$T^\up_{d,C} \bg$ of POD weights $\bg$ are again POD weights. Also 
the answer to this question is in general negative, as the 
following example illustrates. 

Let $d := 3$ and $C := 1$. Moreover, let $\bg$ be POD weights
with $\gamma_{\{1,2,3\}} > 0$. Suppose that the weights 
$\bg^\up := T^\up_{d,C} \bg$ are again POD weights, which 
means that they can be represented in the form 
\[
 \gamma^\up_u =
 \widetilde{\Gamma}_{|u|} \prod_{j\in u} \widetilde{\gamma}_j
\]
for some
$\widetilde{\gamma}_1,\widetilde{\gamma}_2,\widetilde{\gamma}_3 > 0$ 
and
$\widetilde{\Gamma}_0,\dots,\widetilde{\Gamma}_3 > 0$. We obtain
\[
\gamma_{\{1,2\}}^\up =
 \widetilde{\Gamma}_2 \widetilde{\gamma}_1 \widetilde{\gamma}_2 = 
\gamma_{\{1\}}^\up \gamma_{\{2\}}^\up
\frac{\widetilde{\Gamma}_2}{(\widetilde{\Gamma}_1)^2} 
\quad\text{and}\quad
\gamma_{\{1,3\}}^\up =
 \widetilde{\Gamma}_2 \widetilde{\gamma}_1 \widetilde{\gamma}_3 = 
\gamma_{\{1\}}^\up \gamma_{\{3\}}^\up
\frac{\widetilde{\Gamma}_2}{(\widetilde{\Gamma}_1)^2}, 
\]
which implies 
\begin{equation}\label{g55}
\frac{\gamma_{\{1,2\}}^\up}{\gamma_{\{2\}}^\up} 
= 
\frac{\gamma_{\{1,3\}}^\up}{\gamma_{\{3\}}^\up}.
\end{equation}

Specifically, for 
\begin{gather*}
 \gamma_1:=3,\ \gamma_2:=2,\ \gamma_3:=1,\\
 \Gamma_1:=3,\ \Gamma_2:=4,\ \Gamma_3:=5,
\end{gather*}
we obtain 
\[
\gamma_{\{2\}} = 6,\ 
\gamma_{\{3\}} = 3,\ 
\gamma_{\{1,2\}} = 24,\ 
\gamma_{\{1,3\}} = 12,\ 
\gamma_{\{2,3\}} = 8,\ 
\gamma_{\{1,2,3\}} = 30,\ 
\]
and therefore
\[
\gamma_{\{2\}}^\up = 68,\ \gamma_{\{3\}}^\up = 53,\
\gamma_{\{1,2\}}^\up = 54,\ \gamma_{\{1,3\}}^\up = 42,
\]
so that \eqref{g55} is not satisfied. We conclude that the weights
$\bg^\up$ are no longer POD weights.

Regarding $T^\lo_{d,C}$, it is also not hard to see that
for product weights according to (i) with
$\bg \in \M_d$ the transformed weights 
$\bg^\lo := T^\lo_{d,C} \bg$
are in general not even of POD form.
Indeed, 
let $d \geq 3$ and $\gamma_1 := 1$, $\gamma_2 := 1/2$, and
$\gamma_j := 0$ for $j \geq 3$.
Due to Proposition~\ref{l5}, we have $\bg \in \M_d$ and
\[
\left(\Delta_{[s] \setminus u} \bg\right)_u = 
\gamma_u \prod_{j\in [s] \setminus u} (1-\gamma_j) 
\]
for $s \in [d]$ and $u \in \U_d$. Therefore
\[
\gamma_u^{\lo} = 
C^{- 2|u|} \gamma_u 
\prod_{j\in [d]\setminus u} (1-\gamma_j)
\]
in both cases $d \in \N$ and $d = \infty$.
Assume now that the weights $\bg^\lo$ are of POD form, i.e.,
\[
\bg^\lo_u = \Gamma^\prime_{|u|} \prod_{j \in u} \gamma^\prime_j 
\]
for every $u \in \U_d$ with suitably chosen $\Gamma^\prime_k \geq
0$ and $\gamma_j^\prime \geq 0$.
Since $\gamma^\lo_{\{1\}} = C^{-2} /2$ 
and $\gamma^\lo_{\{2\}} = 0$, we conclude that
$\Gamma^\prime_1 > 0$ and $\gamma^\prime_2 = 0$.
It follows that $\gamma^\lo_u = 0$ whenever $2 \in u$.
However, $\gamma^\lo_{\{1,2\}} = C^{-4} /2$,
which is a contradiction.
\end{rem}

Therefore we are aiming at weaker results for product and for POD
weights.

\begin{prop}\label{l6}
Consider product weights $\bg$ according to (i) and any $C>0$.
If $\bg \in \SSS_{d,C}$ then 
\[
\be \leq T^\up_{d,C} \bg \leq c \cdot \be
\]
for the product weights $\be \in \SSS_{d,C}$ given by
\[
\eta_u := C^{2|u|} \gamma_u
\]
for $u\in\U_d$
and
\[
c := \prod_{j \in [d]} (1+C^2 \gamma_j) \in {[1,\infty)}.
\]
If $\bg \in \M_d$ then 
\[
c^\prime \cdot \bz \leq T^\lo_{d,C} \bg \leq \bz
\]
for the product weights $\bz \in \W_d$ given by
\[
\zeta_u := C^{-2|u|} \gamma_u
\]
for $u\in\U_d$ and
\[
c^\prime := \prod_{j \in [d]} (1-\gamma_j) \in [0,1].
\]
Furthermore, $\bz \in \SSS_{d,C}$ if and only if $\bg \in
\SSS_{d,C}$.
\end{prop}

\begin{proof}
For product weights $\bg \in \SSS_d$ we put 
$\bg^{\up} := T^\up_d \bg$. 
For any $u\in\U_d$,
\[
\gamma_u^{\up} = 
\eta_u \prod_{j \in [d] \setminus u} (1+C^2 \gamma_j),
\]
see Remark~\ref{r2}.
Furthermore,
\[
1 \leq \prod_{j \in [d] \setminus u} (1+C^2 \gamma_j) \leq c,
\]
which yields the bounds for $\bg^\up$ as claimed.
Obviously, the weights $\be$ are of product form,
and Proposition~\ref{l4} yields $c < \infty$ as well as $\be \in \SSS_d$.

For product weights $\bg \in \M_d$ we put $\bg^{\lo} := T^\lo_d \bg$. 
As shown in Remark~\ref{r2}, we have 
\[
 \gamma^\lo_u = C^{-2|u|}  \gamma_u \, 
\prod_{j\in [d]\setminus u} (1-\gamma_j).
\]
Proposition~\ref{l5} yields
\[
0 \leq c^\prime 
\leq \prod_{j\in [d]\setminus u} (1-\gamma_j) \leq 1.
\]
Hereby we obtain the bounds for $\bg^\lo$ as claimed, and obviously
the weights $\bz$ are of product form.
Proposition~\ref{l4} yields the equivalence of $\bz \in \SSS_d$
and $\bg \in \SSS_d$.
\end{proof}

\begin{rem}\label{Rem5}
We have a non-trivial lower bound for $T_{d,C}^\lo \bg$ in
Proposition~\ref{l6}, i.e., $c^\prime > 0$, if and only if
$0 \leq \gamma_j < 1$ for all $j \in [d]$ and $\bg$ is summable.
\end{rem}

\begin{prop}\label{l7}
Consider POD weights $\bg$ with $a$
according to (ii) and any $C>0$. We put
\[
c := 
\sum_{w \in \U_d} (|w|!)^a \prod_{j\in w} (2^a C^2 \gamma_j) \in
[1,\infty].
\]
If $\bg \in \SSS_{d,C}$ then
\[
\be \leq T^\up_{d,C} \bg \leq c \cdot \bxi
\]
for the POD weights $\be,\bxi \in \W_d$
given by 
\[
\eta_u := C^{2|u|} \gamma_u
\]
and
\[
\xi_u := 
C_a \left( |u|! \right)^a \prod_{j\in u} 
\left( 2^a C^2 \gamma_j \right)
= C_a \left( 2^{|u|} |u|! \right)^a \eta_u. 
\]
If $\gamma \in \M_d$ then we have
\[
\bn \leq T^\lo_{d,C} \bg \le \bz 
\]
for the POD weights $\bz \in \W_d$ given by
\[
\zeta_u := C^{-2|u|} 
\gamma_u.
\]
\end{prop}

\begin{proof}
For POD weights $\bg \in \SSS_d$
the weights $\bg^{\up} := T^\up_d \bg$ take the form
\begin{align*}
\gamma_u^{\up} &= 
\sum_{\substack{v \in \U_d\\ u \subseteq v}} C^{2|v|} \gamma_v 
= \sum_{\substack{v \in \U_d\\ u \subseteq v}}
\Gamma_{|v|} \prod_{j\in v}
\left(C^2\gamma_j\right)\\
&= \prod_{j\in u} \left(C^2\gamma_j\right) 
\sum_{\substack{w \in \U_d\\ u \cap w = \emptyset}}
\Gamma_{|u|+|w|} \prod_{j\in w}
\left(C^2\gamma_j\right).
\end{align*}
This gives us 
\[
\gamma_u^{\up} \ge C^{2|u|} \Gamma_{|u|} \prod_{j\in u} \gamma_j 
=
\eta_u,
\]
which yields the lower bound $\bg^\up \ge \be$.
For an upper bound on the terms $\gamma_u^{\up}$,
we first estimate
\[
\gamma_u^{\up}  \le C_a  \prod_{j\in u} \left(C^2\gamma_j\right) 
\sum_{\substack{w \in \U_d\\ u \cap w = \emptyset}}
\left((|u|+|w|)!\right)^a \prod_{j\in w}
\left(C^2\gamma_j\right)
\]
and then follow the lines of \citet[pp.~208--209]{DKS13}. Indeed,
\[
\gamma_u^{\up} 
\le C_a (|u|!)^a \prod_{j\in u} \left(C^2\gamma_j\right) 
\cdot \sum_{\substack{w \in \U_d\\ w \cap u = \emptyset}}
\frac{\left((|u|+|w|)!\right)^a}{\left(|u|!\right)^a}
\prod_{j\in w} \left(C^2\gamma_j\right).
\]
Since
\[
\frac{(|u|+|w|)!}{|u|!} = 
\binom{|u|+|w|}{|w|} \cdot |w|! \leq 
2^{|u|+|w|} \cdot |w|!,
\]
we obtain
\[
\gamma_u^{\up} 
\leq 
C_a (|u|!)^a \prod_{j\in u} \left(2^a C^2\gamma_j\right)
\cdot
\sum_{w \in \U_d} \left(|w|!\right)^a
\prod_{j\in w} \left(2^a C^2 \gamma_j\right).
\]
Consequently, $\gamma_u^{\up} \le c \,\xi_u$.

Now let us turn to weights $T^\lo_d \bg$ induced by some 
$\bg \in \M_d$: The upper bound follows by the definition of 
$T_d^\lo$ and Remark~\ref{r1}, the lower bound is obvious. 
\end{proof}

\begin{rem}\label{Rem6}
The lower bound for $T^\lo_{d,C} \bg$ in Proposition~\ref{l7}
is sharp in the following 
sense. Consider the case $d=2$, and POD weights with 
\begin{gather*}
\gamma_1 := 1,\ \gamma_2 :=1/2,\\
\Gamma_1 := 1/2,\ \Gamma_2=1. 
\end{gather*}
Then we have, e.g., 
\[
\left(T^\lo_{d,C} \bg\right)_{\{1\}}
= C^{-2} \gamma_{\{1\}}- C^{-2} \gamma_{\{1,2\}}
= C^{-2}\Gamma_1 \gamma_1 - C^{-2} \Gamma_2 \gamma_1 \gamma_2 = 0. 
\]
 
Similar examples can be constructed for other choices of $d$ and $u$, 
respectively. 
\end{rem}

\begin{rem}
Let $c$ and $\bxi$ be defined as in Proposition~\ref{l7}.
Let us comment briefly on the condition $c<\infty$ in the
case $d=\infty$, which
leads to a non-trivial upper bound for $T^\up_{d,C} \bg$ in
Proposition~\ref{l7}. 
Note that $c < \infty$ is equivalent to the summability of the POD
weights $\bxi$, i.e., to $\bxi \in \SSS_{d,1}$.
According to Proposition~\ref{l4}, a sufficient condition for the
latter is $p > a$ and $\sum_{j \in \N} \gamma_j < \infty$.
\end{rem}

\subsection{The Decay of Transformed Weights}\label{Subsec:Decay}

In the sequel, we consider the case $d=\infty$.
As before, we use the notation
$\bg^{1/\tau} = (\gamma_u^{1/\tau})_{u \in \U_d}$ for 
$\bg \in \W_d$ and $\tau > 0$. 

We relate the decay of the transformed weights
$T^\up_{d,C} \bg$ to the decay of the original weights $\bg 
\in \SSS_{d,C}$. 
The relevance of the results from the present section
will be indicated 
in Remark~\ref{r21} below.

At first we point to an extremal case.

\begin{rem}\label{rem:dec_extremal_case}
For $u \in \U_d$ we put $\gamma_u := 2^{-j}$ if $u =
\{1,\dots,2^j\}$ for any $j \in \N$ and $\gamma_u := 0$
otherwise. Obviously we have $\decay(\bg) =\infty$. Let $C := 1$
as well as $j \in \N$ and 
$\{1,\dots,2^j\} \subseteq v \subseteq \{1,\dots,2^{j+1}\}$, 
so that $(T^\up_{d,C} \bg)_v \geq 2^{-(j+1)}$. Consequently, 
\[
\sum_{u \in \U_d} \left(T^\up_{d,C} \bg\right)^{1/\tau}_u
\geq \sum_{j \in \N} 2^{2^j} \cdot 2^{-(j+1)/\tau} = \infty
\]
for every $\tau>0$. Therefore we have $\decay(T^\up_{d,C} \bg) = 0$.
\end{rem}

Let us now check the decay of the transformed weights $T^\up_{d,C} \bg$ 
for the particular classes (i)--(iii) of weights studied before.
To this end, we note that if $\bg$ belongs to one of these
classes, then $\bg^{1/\tau}$ belongs to the same class. 

\begin{prop}\label{le:decay_POD}
Consider POD weights $\bg$ according to (ii) with $\Gamma_1 >0$, 
$p > a$, and $\sum_{j\in\N} \gamma_j < \infty$, and any $C>0$.
Then we have
\[
\decay \left(T^{\up}_{d,C} \bg\right) = \decay (\bg) = p.
\]
\end{prop}

\begin{proof}
At first, we note that
$\decay ((\gamma_j^{1/\tau})_{j\in\N}) = p / \tau$ for
every $\tau > 0$. 
Let us consider the POD weights $\bg^{1/\tau}$ for any $\tau > 0$.
We first show $\decay(\bg) = p$. 

Due to $\gamma_{\{j\}} = \Gamma_1 \gamma_j$, we obtain 
\[
\sum_{u\in\U_d} \gamma_u^{1/\tau} \ge 
\Gamma_1^{1/\tau} \sum_{j\in\N} \gamma_j^{1/\tau}
\] 
for all $\tau >0$. Since $\Gamma_1 >0$, this yields $\decay(\bg) \le p$. 

Let $\tau \in (0, p)$. Then $\sum_{j\in\N} \gamma_j^{1/\tau} < \infty$, 
and our assumption $p>a$ implies $p/\tau > a/\tau$. Hence, due to 
Lemma~\ref{prel4}, there exists a constant $c >0$ such that 
\[
\sum_{\substack{u \subseteq [j] \\ j\in u}}  \gamma_u^{1/\tau} \le 
c  \cdot \gamma_j^{1/\tau}
\]
for all $j\in\N$. This implies $\decay(\bg) \ge p$. 

Due to Proposition~\ref{l4} we have 
$\bg \in \SSS_{d, C}$, hence $\bg^\up :=
T^\up_{d, C} \bg$ is well-defined.
Next we show $\decay(\bg^\up) \le p$.
Let $\tau \in (0,\decay(\bg^\up))$. Due to the first implication 
of Lemma~\ref{l3} we have $\bg^{1/\tau} \in \SSS_{d,C^{1/\tau}}$.
Now Proposition~\ref{l4}
yields $\sum_{j\in\N} \gamma_j^{1/\tau} < \infty$, hence
$p \geq \decay(\bg^\up)$.

Now we use our assumption $\sum_{j\in\N} \gamma_j < \infty$, 
which 
implies, in particular, $p\ge 1$. Let $\tau \in [1,p)$ if $p>1$ and let 
$\tau = p =1$ otherwise.
In both cases we have $\sum_{j\in\N} \gamma_j^{1/\tau} < \infty$.
According to Proposition~\ref{l4}, 
$\bg^{1/\tau} \in \SSS_{d,\sqrt{2} C^{1/\tau}}$ holds.
Due to the second implication from
Lemma~\ref{l3} we obtain $\decay(\bg^\up) \geq p$.
This concludes the proof of the lemma.
\end{proof}

Since product weights $\bg$ according to (i) are, in particular, POD 
weights according to (ii) with $\Gamma_k = 1$ for all $k\in [d]
\cup \{0\}$ and with any $a > 0$,
Proposition~\ref{le:decay_POD}
immediately yields the following corollary.

\begin{corollary}\label{le:decay_prod}
Consider product weights $\bg$ according to (i)
with $\sum_{j \in \N} \gamma_j < \infty$ and any $C>0$.
Then we have
\[
\decay \left(T^{\up}_{d,C} \bg\right) = \decay (\bg) = p.
\]
\end{corollary}

We also have the following result.

\begin{prop}\label{le:decay_FO}
Consider finite-order weights $\bg$ according to 
(iii) and any $C > 0$. If $\bg \in \SSS_{d,C}$ then
\[
 \decay \left(T^{\up}_{d,C} \bg \right) =\decay (\bg)
 \ge 1.
\]
\end{prop}

\begin{proof}
In addition to $\bg$, we consider the finite-order weights 
$\bg^{1/\tau}$ for any $\tau > 0$.
For every $\widetilde{C}>0$ Proposition~\ref{l4} shows that 
$\bg^{1/\tau} \in \SSS_{d,\widetilde{C}}$ is equivalent to the
summability of $\bg^{1/\tau}$.

Let us now assume that $\bg \in \SSS_{d,C}$. 
We combine the latter equivalence and Lemma~\ref{l3} to
establish the following facts for $\bg^\up := T^\up_d \bg$: 
Summability holds for $\bg$ and
$\bg^\up$, and for every $\tau>1$ the summability of $\bg^{1/\tau}$
and $\left(\bg^\up\right)^{1/\tau}$ are equivalent.
Therefore $\decay(\bg^\up) = \decay(\bg)$.
\end{proof}

\section{Function Space Embeddings}\label{Sec:Embeddings}

Now we study embeddings of reproducing kernel Hilbert spaces
(RKHSs); for elementary facts on RKHSs and 
their kernels we refer to \citet{Aro50} 
and \citet{PR16}. For a reproducing kernel $M \colon \Z
\times \Z \to \K$ we 
denote its uniquely determined RKHS by
$H(M)$, the corresponding norm by $\|\cdot\|_{M}$,
and we say that $M$ is a reproducing kernel over the domain $\Z$.

At first, we state two basic facts concerning embeddings of RKHSs.
For reproducing kernels $M$ and $M^\prime$ over a common domain $\Z$
we say that $H(M)$ is contractively contained in $H(M^\prime)$,
if $H(M) \subseteq H(M^\prime)$ with an embedding of norm at most
one. 
Furthermore, we denote by $M \otimes M^{\prime}$ the function
given by 
\[
M \otimes M^{\prime} \big( (x_1, x_2), (y_1, y_2) \big) :=
M(x_1, y_1) \cdot  M^{\prime} (x_2, y_2)
\]
for all $(x_1, x_2), (y_1, y_2) \in \Z\times \Z$.
It follows immediately from Schur's product theorem,
see, e.g., \citet[Thm.~4.8]{PR16}, that 
$M \otimes M^{\prime}$ is again a reproducing kernel.

For the following result see, e.g., \citet[Cor.~5.3]{PR16}.

\begin{lemma}\label{l10}
The space $H(M)$ is contractively contained in $H(M^\prime)$ if and only 
if $M^\prime-M$ is a reproducing kernel.
\end{lemma}

The following result, which deals with tensor products, is easily
derived from Lemma~\ref{l10}.

\begin{lemma}\label{l11}
If $H(M_i)$ is contractively contained in $H(M^\prime_i)$
for $i=1,2$, then $H(M_1 \otimes M_2)$ is contractively contained
in $H(M^\prime_1 \otimes M^\prime_2)$.
\end{lemma}

Next, we consider reproducing kernels $M$ and $M^\prime$ over 
domains $\Z$ and $\Z^\prime$,
respectively, such that $\Z^\prime \subseteq \Z$.
We write 
\[
H(M) \sqsubseteq H(M^\prime)
\] 
if $f|_{\Z^\prime} \in H(M^\prime)$ for 
every $f\in H(M)$. The restriction theorem 
yields the following generalization of  
Lemma~\ref{l10}. 

\begin{lemma}\label{l12}
We have $H(M) \sqsubseteq H(M^\prime)$ with a restriction mapping
$f \mapsto f|_{\Z^\prime}$ of norm at most one, if and only if
$M^\prime-M|_{\Z^\prime \times \Z^\prime}$ is a reproducing kernel.
\end{lemma}

\begin{proof}
The restriction theorem says that $H(M|_{\Z^\prime \times
\Z^\prime})$ consists of all functions $f|_{\Z^\prime}$ with
$f \in H(M)$,
and
\[
\|g\|_{M|_{\Z^\prime \times \Z^\prime}} = \inf\{\|f\|_M : 
\text{$f \in H(M)$, $f|_{\Z^\prime} = g$}\}
\]
for every $g \in H(M|_{\Z^\prime \times \Z^\prime})$, see, e.g.,
\citet[Cor.~5.8]{PR16}. 

Suppose that $H(M) \sqsubseteq H(M^\prime)$ with a restriction mapping
of norm at most one. Let
$g \in H(M|_{\Z^\prime \times \Z^\prime})$. We obtain $g \in
H(M^\prime)$ and $\|g\|_{M^\prime} \leq \|f\|_M$
for every $f \in H(M)$ with $f|_{\Z^\prime} = g$. Therefore
\[
\|g\|_{M^\prime} \leq 
\|g\|_{M|_{\Z^\prime \times \Z^\prime}},
\]
which shows that
$H(M|_{\Z^\prime \times \Z^\prime})$ is
contractively contained in $H(M^\prime)$.
Due to Lemma~\ref{l10} we have that 
$M^\prime-M|_{\Z^\prime \times \Z^\prime}$ is a reproducing kernel.

Conversely, suppose that the latter holds true.
Due to Lemma~\ref{l10}  we obtain for
$f \in H(M)$ that $f|_{\Z^\prime} \in H(M^\prime)$ and
\[
\|f|_{\Z^\prime}\|_{M^\prime} \leq 
\|f|_{\Z^\prime}\|_{M|_{\Z^\prime \times \Z^\prime}} \leq \|f\|_M
\] 
holds, which shows that $H(M) \sqsubseteq H(M^\prime)$ with a 
restriction mapping of norm at most one. 
\end{proof}

\subsection{Superpositions of Weighted Tensor Product Kernels}
\label{s4.1}

We study embeddings and restriction mappings 
between RKHSs of functions of $d$ variables,
where $d \in \N \cup \{\infty\}$. The corresponding reproducing
kernels are assumed to be superpositions of weighted tensor product
kernels.

The starting point for the construction of 
this kind of spaces is given by a 
reproducing kernel $m \neq 0$ over an arbitrary domain
$D \neq \emptyset$.
The elements of $H(1+m)$ 
are considered to be functions of a single variable $x \in D$.
For $u\in \U_d$ the tensor product kernel $m_{u}$ is defined by
\[
m_{u}(\bx,\by) := \prod_{j\in u} m(x_j, y_j)
\]
for all $\bx, \by\in D^{[d]}$,
where by convention $m_{\emptyset} = 1$. 
The functions $f \in H(m_u)$, which are formally defined
on $D^{[d]}$, may be identified with functions of $|u|$ variables,
since $f(\bx)$ depends only on the coordinates 
$x_j$ with $j\in u$. 

In addition to the reproducing kernel $m$, we consider weights 
$\be \in \W_d$ such that 
\[
\Z^{\be,m} := \biggl\{\bx \in D^{[d]} : 
\sum_{u\in \U_d} \eta_u m_u(\bx, \bx) <\infty \biggr\}
\]
is non-empty.
This allows us to define a reproducing kernel $M^{\be,m}$ by
\[
M^{\be,m} (\bx, \by) := \sum_{u\in\U_d} \eta_u m_u(\bx, \by)
\]
for all $\bx, \by \in \Z^{\be,m}$. The elements of $H(M^{\be,m})$ are 
functions of $d$ variables.

\begin{rem}\label{r4}
In the case $d \in \N$ we obviously have 
$\Z^{\be,m} = D^{[d]} = D^d \neq \emptyset$.

In the case $d = \infty$ we have
\[
\Z^{\be,m} \neq \emptyset \quad \Leftrightarrow \quad
\sum_{u \in \U_d} q^{|u|} \eta_u < \infty
\]
with $q := \inf_{x \in D} m(x,x)$,
and $\Z^{\be,m}$ is the maximal domain for the 
reproducing kernel $M^{\be,m}$. See \citet[Sec.~2]{GMR12} for
details.  
For the discussion of the particular cases of product weights, POD
weights, and finite-order weights we refer to 
Proposition~\ref{l4} with $C := q^{1/2}$.
\end{rem}

Actually, we consider a pair of reproducing kernels 
$m,m^\up$ over the domain $D$,
a pair of weights $\be, \be^\up \in \W_d$,
and a constant $C > 0$
with the following properties:
\begin{itemize}
\item[(A1)]
We have $m,m^\up \neq 0$ and
\[
H(m)\subseteq H(1+m^\up).
\]
Moreover,
\[
C \geq
\sup \{ \|f\|_{1+m^\up} : f \in H(m),\ \|f\|_m \leq 1\}.
\]
\item[(A2)]
We have $\be \in \SSS_{d,C}$ and
\[
\be^{\up} = T_{d,C}^\up \be.
\]
\item[(A3)]
We have 
$\Z^{\be^\up\!,m^\up} \neq \emptyset$.
\end{itemize}

\begin{rem}\label{r5}
Observe that $H(m) \neq \{0\}$ is
continuously embedded into $H(1+m^\up)$ 
due to the closed graph theorem.
Consequently, the constant $C$ may be any upper bound on 
the norm of this embedding. It follows that
$H(m)$ is contractively contained in $H(C^2(1+m^\up))$.
\end{rem}

\begin{rem}\label{r6}
In the case $d \in \N$ the assumption (A3) is trivially satisfied,
since $\Z^{\be^{\up},m^{\up}} = \Z^{\be,m} = D^{[d]} = D^d \neq 
\emptyset$, cf.~Remark~\ref{r4}.

Consider the case $d = \infty$. For every $q \geq 0$ we have
\[
\sum_{u \in \U_d} q^{|u|} \eta_u^\up =
\sum_{v \in \U_d} \sum_{\substack{u \in \U_d \\ u \subseteq v}}
q^{|u|} C^{2|v|} \eta_v =
\sum_{v \in \U_d} (1+q)^{|v|} C^{2|v|} \eta_v.
\]
Using Remark~\ref{r4}, we conclude that (A3) is equivalent to
$\be \in \SSS_{d,\widetilde{C}}$ with
\[ 
\widetilde{C} := (1+ \inf_{x \in D} m^\up(x,x) )^{1/2} \cdot C.
\]
Since $\widetilde{C}^2 \geq \inf_{x \in D} m(x,x)$, which 
follows from
Remark~\ref{r5}, we may use Remark~\ref{r4} again to conclude that
(A3) implies $\Z^{\be,m} \neq \emptyset$.
\end{rem}

Given (A1)--(A3), we study the reproducing kernels 
\[
M := M^{\be,m}
\hspace{3ex}\text{and}\hspace{3ex}
M^\up := M^{\be^\up\!,m^\up}
\]
over the corresponding maximal domains  
\[
\Z := \Z^{\be,m}
\hspace{3ex}\text{and}\hspace{3ex}
\Z^\up := \Z^{\be^\up\!,m^\up},
\]
respectively.

\subsection{Functions of Finitely Many Variables}
\label{Subsec:Finite_Variables_FSE}

We consider the setting described in Section~\ref{s4.1}
and confine ourselves in this subsection to 
the case $d\in \N$, where $H(M)$ and 
$H(M^{\up})$ consist of functions over the common 
domain $\Z = \Z^{\up} = D^d$.

\begin{theo}\label{Lemma1}
Let $d \in \N$. If {\rm (A1)} and {\rm (A2)} are satisfied, then
the space $H(M)$ is contractively contained in $H(M^{\up})$.
\end{theo}

\begin{proof}
Put
\[
\widehat{m}_u(\bx,\by) := 
\prod_{j\in u} \left(C^2(1+m^{\up}(x_j,y_j)\right)
= C^{2|u|} \sum_{\substack{v \in \U_d \\ v \subseteq u}}
m^{\up}_{v}(\bx,\by)
\]
for $u \in \U_d$ and $\bx,\by \in D^d$.
First note that
\begin{equation*}
\begin{split}
M^{\up}
&=\sum_{v\in \U_d} \eta_v^{\up} \, m^{\up}_{v}
=\sum_{v\in \U_d} \left( \sum_{\substack{u \in \U_d\\  v \subseteq u }}
C^{2 |u|}\eta_u \right) m^{\up}_{v} \\
&=\sum_{u\in \U_d}C^{2\left|u\right|}\eta_u
\sum_{\substack{v \in \U_d \\ v\subseteq u}}m^{\up}_{v}
= \sum_{u\in \U_d}\eta_u \widehat{m}_u.
\end{split}
\end{equation*}
Therefore
\[
M^{\up} - M =
\sum_{u\in \U_d}\eta_u \left(\widehat{m}_u - m_u \right).
\]
Since $H(m)$ is contractively contained in $H(C^2(1+m^{\up}))$, 
see Remark~\ref{r5}, 
the same holds true for the spaces $H(m_{u})$ and $H(\widehat{m}_u)$,
see Lemma~\ref{l11}. Hence Lemma~\ref{l10} yields the claim.
\end{proof}

\subsection{Functions
of Infinitely Many Variables}
\label{Subsec:Infinite_Variables_FSE}

In the sequel, we consider again the setting described in 
Section~\ref{s4.1}, and focus in this subsection on the case 
$d = \infty$. We have the following kind of embedding 
or, more precisely, restriction result.

\begin{theo}\label{Lemma4}
Let $d = \infty$. If {\rm (A1)--(A3)} are satisfied, then we have
$\Z^{\up} \subseteq \Z$ and
$$
H(M) \sqsubseteq H (M^{\up}).
$$ 
Moreover,
the norm of the corresponding restriction mapping 
\[
H(M) \to H (M^{\up}),
f \mapsto f|_{\Z^{\up}},
\]
is at most one.
\end{theo}

\begin{proof}
We proceed as in the proof of Theorem~\ref{Lemma1}, and to do so
we only have to verify some summability conditions.
Again we consider
\[
\widehat{m}_u(\bx,\by) := 
\prod_{j\in u} \left(C^2(1+m^{\up}(x_j,y_j)\right)
= C^{2|u|} \sum_{\substack{v \in \U_\infty \\ v \subseteq u}}
m^{\up}_{v}(\bx,\by)
\]
for $u \in \U_\infty$ and $\bx,\by \in D^\N$.
First of all, we note that
\[
\sum_{v\in \U_\infty} \eta_v^{\up} \, |m^{\up}_{v} (\bx,\by)| =
\sum_{v\in \U_\infty} 
\sum_{\substack{u \in \U_\infty\\  v \subseteq u }}
C^{2 |u|}\eta_u \, |m^{\up}_{v} (\bx,\by)| < \infty
\]
for $\bx,\by \in \Z^{\up}$. Indeed, for $\bx=\by$ we employ the
definition of $\Z^\up$, and in the general case we use
the Cauchy-Schwarz inequality together with
\[
|m^{\up}_{v} (\bx,\by)|^2 \leq
m^{\up}_{v} (\bx,\bx) \cdot
m^{\up}_{v} (\by,\by).
\] 
Thus we have 
\begin{align}\label{eq:for_next_remark}
M^{\up} (\bx,\by) 
&= \sum_{v\in \U_\infty} \eta_v^{\up} \, m^{\up}_{v} (\bx,\by) 
=\sum_{v\in \U_\infty} \left( 
\sum_{\substack{u \in \U_\infty\\  v \subseteq u }}
C^{2 |u|}\eta_u \right) m^{\up}_{v} (\bx,\by)  \notag \\
&=\sum_{u\in \U_\infty}C^{2\left|u\right|}\eta_u
\sum_{\substack{v \in \U_\infty \\ v\subseteq u}} m^{\up}_{v}(\bx,\by) 
= \sum_{u\in \U_\infty} \eta_u \widehat{m}_u (\bx,\by)
\end{align}
with absolutely convergent series.

Now, due Remark~\ref{r5},  
we know that  
$m(x,x) \le C^2 (1+ m^{\up}(x,x))$ 
for all $x\in D$. This implies for all $\bx \in \Z^{\up}$ that
\[
\sum_{u\in \U_\infty} \eta_u m_u  (\bx, \bx)
\le \sum_{u\in \U_\infty} \eta_u \widehat{m}_u  (\bx, \bx)
\le \sum_{v\in \U_\infty} \eta_v^{\up} m_v^{\up}  (\bx, \bx)
< \infty,
\]
i.e., we have $\Z^{\up} \subseteq \Z$.
Moreover, we obtain for all $\bx, \by \in \Z^{\up}$ that
\[
\left( M^{\up} - M \right) (\bx, \by) =
\sum_{u\in \U_\infty} \eta_u \left(\widehat{m}_u - m_u \right) (\bx, \by)
\]
and the sum converges absolutely. 
As already stated in the proof of Theorem~\ref{Lemma1}, 
the fact that $H(m)$ is contractively contained in 
$H(C^2(1+m^{\up}))$ implies that 
the same holds true for the spaces 
$H(m_{u})$ and $H(\widehat{m}_u)$.
This implies that $\widehat{m}_u - m_u$ is a reproducing kernel, 
cf.\ Lemma~\ref{l10}, and this property is inherited by 
$\left( \widehat{m}_u - m_u \right)|_{\Z^{\up}\times \Z^{\up}}$.
From that it is easily seen that 
$M^{\up} - M|_{\Z^{\up}\times \Z^{\up}}$ is a reproducing kernel. 
Now Lemma~\ref{l12} yields the statement of the theorem. 
\end{proof}

In the final remark of this section we demonstrate that the 
domains $\Z$ and $\Z^{\up}$ can indeed be different. 

\begin{rem}
We again consider the weights from 
Remark~\ref{rem:dec_extremal_case} in the case
$d=\infty$:
For $u \in \U_d$ we put $\eta_u := 2^{-j}$ if $u =
\{1,\dots,2^j\}$ for any $j \in \N$ and $\eta_u := 0$ otherwise.
Furthermore, we consider the kernel $m$ on the unit interval 
$D := [0,1]$ given by 
$m(x,y) = 1$ if $x=y\in (0,1]$ and $m(x,y) = 0$ 
otherwise. We take $m^{\up}:= m$
and may therefore choose $C=1$. 
Then, on the one hand, we have for all $\bx \in [0,1]^\N$ that
\[
M(\bx,\bx) 
= \sum_{u\in \U_\infty} \eta_u \, m_{u} (\bx,\bx) 
= \sum_{j\in \N} 2^{-j} \prod_{\nu =1}^{2^j}  m(x_\nu, x_\nu)
\le \sum_{j\in \N} 2^{-j} < \infty.
\]
This shows that $\Z = [0,1]^\N$.  On the other hand, 
due to~\eqref{eq:for_next_remark},
\[
M^{\up} (\bx,\bx) 
= \sum_{u\in \U_\infty} \eta_u \widehat{m}_u (\bx,\bx)
= \sum_{j\in \N} 2^{-j} \prod_{\nu =1}^{2^j} 
\big( 1 + m^{\up}(x_\nu, x_\nu) \big)
=  \sum_{j\in \N} 
2^{a_j(\bx)-j},
\]
where
\[
a_j(\bx) := | \{ \nu \in \{1,\dots,2^j\} : x_\nu > 0\}.
\]
This shows that 
\[
\Z^{\up} \subsetneq 
\{ \bx \in [0,1]^\N :
\exists  j_0 \in \N \, \forall j \ge j_0: 
a_j(\bx)
\le j\}
\subsetneq [0,1]^\N = \Z.
\]
\end{rem}

\section{Application to Computational Problems}%
\label{Sec:Application}

Now we describe how to apply our findings from the previous 
sections to the study of the complexity of continuous computational
problems.

In tractability analysis we are interested in continuous 
computational problems that are defined on a scale of function 
spaces $H := H_d$ with $d \in \N$, where the parameter
$d$ typically denotes the number of variables the functions in 
$H_d$ depend on. An important goal is to find algorithms that scale 
well with respect to the 
dimension parameter $d$ and help us to identify whether 
the curse of dimensionality can be broken in a given setting. 

Tractability analysis still deals with a finite (although 
arbitrarily large) number of variables. Now many applications 
rely on models that depend on an infinite (but still countable)
number of variables,
which motivates the study of spaces $H := H_\infty$ of functions
that depend on infinitely-many variables.
Examples include stochastic models that are based 
on infinite sequences of independent random variables.
The complexity analysis of the resulting continuous problems may 
be viewed as the limit of tractability analysis for 
$d$-variate functions, where $d$ tends to infinity.

Let $d \in \N \cup \{\infty\}$.
Some specific features of the function spaces can be very helpful 
for the analysis of certain types of algorithms. For instance, for 
the analysis of unbiased randomized algorithms it may be extremely 
helpful if the norm on $H_d$ induces an ANOVA
decomposition on $H_d$, which can be used to analyze 
the error ($=$ variance) of the algorithms see, e.g., \citet{BG14}, 
in particular, Lemma~2.1. Contrarily, for the analysis of 
deterministic algorithms it may be helpful to use the Taylor 
expansion of input functions in certain points and therefore 
it is beneficial to work with an underlying anchored function space 
decomposition, see, e.g., \citet[Section~5.3]{DG12}.

In general, a suitable embedding or restriction mapping between
functions spaces may help to transfer results from 
a function space with favorable features to another
function space.
There are several general embedding results known in the case where 
$H_d$ is the $d$-fold tensor product of 
reproducing kernel Hilbert spaces of a single variable;
examples include Hilbert spaces with reproducing
kernels of the form \eqref{g17},
where the weights $\gamma_u$ are product weights,
and tensor products of
spaces of increasing smoothness, see, 
e.g., \citet{HR13}, \citet{GHHR17}, and \citet{GHHRW2019}.

Due to the machinery we have developed in the previous 
sections, we now have a general embedding approach at hand that 
still works in the case where $H_d$ is not necessarily of tensor 
product form. This is, e.g., the case if we consider weighted RKHSs, 
where the weights are not of product form but are, e.g., product 
and order-dependent or 
of finite-order instead, see Section~\ref{SEC:Part_Weights}. 

Let us now consider for $d\in \N \cup \{\infty\}$ the reproducing 
kernel 
\[
K := M^{\bg,k}
\]
as defined in Section~\ref{s4.1} by means of a
\begin{itemize}
\item[(B1)]
reproducing kernel 
$k \neq 0$ over a domain $D$
and weights $\bg \in \W_d$ such that
$\Z^{\bg,k} \neq \emptyset$.
\end{itemize}
\noindent
Suppose that we are, e.g., interested in solving a computational problem 
on the space $H := H(K)$ of functions of $d$ variables, 
further specified by a Banach space $G$ and a continuous 
linear solution operator 
\[
S \colon H(K) \to G.
\]
Solution operators 
of this type arise, e.g., for numerical integration,
function recovery, or well-posed linear operator equations.

The study of computational
problems on weighted RKHSs was initiated in \citet{SW98} for 
$d \in \N$. For more results in the case of finitely 
many variables we refer to the monographs \citet{NW08, NW10, NW12}.
The study of numerical problems on weighted RKHSs of functions with 
an infinite number of variables was initiated in \citet{HW01}, 
\citet{HMNR10}, and \citet{KSWW10a}. More recent results can be 
found in \citet{GHHR17}, and, for closely related spaces, in 
\citet{GHHRW2019, GHRR2024, GRR2024}.

We are allowed to use (deterministic or randomized) linear 
algorithms $A$, and the error criterion we consider is the 
(deterministic or randomized) worst case error
\[
\err \left( A, S, K \right) 
:= \sup_{\|f\|_{K} \le 1} 
\left(\EE \left( \| S f - A f\|^2_{G} \right) \right)^{1/2},
\]
where $\EE$ means 
taking the expected value of the expression in parentheses. 
In the case of deterministic algorithms this simplifies to
\[
\err \left( A, S, K \right) 
:= \sup_{\|f\|_{K} \le 1} 
\| S f - A f\|_{G}.
\]

Assume that for some reason the norm and the underlying function 
decomposition on $H(K)$ induced by the given univariate 
kernel $k$ is not well-suited for the type of error 
analysis one wants to perform, e.g., deriving upper error bounds for 
specific algorithms or lower error bounds valid for whole classes 
of algorithms.  Therefore we choose a different 
\begin{itemize}
\item[(B2)]
reproducing kernel $l \neq 0$ over the domain $D$
that satisfies either
\begin{equation}\label{inclusion_upper_bounds}
H(k) \subseteq H(1+l)
\end{equation}
if we are interested in upper error bounds or
\begin{equation}\label{inclusion_lower_bounds}
H(l) \subseteq H(1+k)
\end{equation}
if we are interested in lower error bounds.
\end{itemize}

We may choose, e.g., an anchored or an ANOVA kernel $l$. For
certain weights $\br \in \W_d$ we may obtain a RKHS 
$H(M^{\br,l})$ with a more favorable function space 
decomposition, so that in this setting we may perform the type of 
error analysis we are interested in. 
We add that the stronger assumptions
$H(k) \subseteq H(l)$ or $H(l) \subseteq H(k)$ are
typically not satisfied in the 
applications to the computational problems we have
in mind.

The question is now: \emph{How to choose the weights 
$\br$ such that we can transfer our results on 
$H(M^{\br,l})$ back to $H(K)$?}
This will be described in the next two subsections, where we 
treat upper and lower error bounds separately.

\subsection{Upper Error Bounds}

Throughout the study of upper error bounds we assume
the following:
The assumptions (B1) and (B2) with inclusion 
\eqref{inclusion_upper_bounds} are satisfied, and
$C^\up > 0$ is an upper bound on the 
norm of the embedding of $H(k)$ into $H(1+l)$.

Assuming that $\bg \in \SSS_{d,C^\up}$,
we simply choose 
\[
\br := \bg^{\up} := T^{\up}_{d,C^\up} \bg,
\]
and assuming that $\Z^{\bg^\up\!,l} \neq \emptyset$ we put
\[
L^\up := M^{\bg^\up,l}.
\]
For product, POD, and finite-order weights,
conditions for $\bg \in \SSS_{d,C}$ to hold for any value of $C>0$
are presented in Proposition~\ref{l4}.

Accordingly, the assumptions (A1)--(A3) are satisfied with 
\[
m := k,\ m^\up := l,\ C := C^\up,\ \be := \bg,\ \be^\up := \bg^\up.
\]
Due to Theorems~\ref{Lemma1} and \ref{Lemma4} we have 
$\Z^{\bg^\up\!,l} \subseteq \Z^{\bg,k}$ and
$H(K) \sqsubseteq H(L^\up)$,
and the corresponding 
(embedding or restriction) map
\[
\mathfrak{i^\up} \colon H(K) \to H(L^\up)
\]
has norm at most one.

Assume that the continuous linear map
\[
S^{\up} \colon H(L^\up) \to G
\]
is a ``natural extension'' of $S$ in the sense that it satisfies 
\[
S = S^{\up} \circ \mathfrak{i^\up}.
\]
For any linear algorithm $A^{\up}$ on $H(L^\up)$ we get a 
``restricted algorithm'' $A$ on $H(K)$ via
\[
A:= A^{\up}  \circ 
\mathfrak{i^\up},
\]
which is again linear. 
Furthermore, if $A^\up$ is based on function evaluations,
then the same holds true for $A$.

\begin{corollary}\label{c1}
If $\bg \in \SSS_{d,C^\up}$ and
$\Z^{\bg^\up\!,l} \neq \emptyset$, then we have
\[
\err(A, S, K) \leq \err(A^{\up}, S^{\up}, L^\up). 
\]
\end{corollary}

Consequently, any upper error bound for $S^{\up}$ and $A^{\up}$ 
on the unit ball in $H(L^\up)$ also holds for $S$ and $A$ on 
the unit ball in $H(K)$.

\begin{rem}\label{r21}
It is important to also observe 
the decay of the transformed weights $\bg^\up$, 
due to its relation to minimal errors for 
computational problems and its appearance in error bounds. 

For the case $d \in \N$, in many problems defined on 
weighted function spaces, the decay is directly related to 
the exponent of strong tractability, which plays a 
prominent role in the 
complexity analysis for high-dimensional
problems, see \citet{NW08,NW10,NW12}.
The inverse of the exponent of strong tractability is (essentially) the 
best possible convergence rate that can appear in upper bounds on minimal 
errors, if we require the bounds to be independent of 
the dimension $d$.  For corresponding error bounds we refer to, e.g., 
\citet[Theorems~5.8 to 5.10, 5.12 and 6.10]{DKS13}, 
and also to numerous theorems in \citet{DP10}, 
\citet{DKP22}, or again to \citet{NW08,NW10,NW12}.

In the case  $d = \infty$, in all results we are aware of, the 
best possible convergence rates of the minimal errors depend 
crucially on the decay of the weights of the underlying RKHS. 
For corresponding results we refer to, e.g.,
\citet{HMNR10}, \citet{KSWW10a}, \citet{Gne10}, 
\citet{PW11}, and \citet{GRR2024}.

This explains the relevance of
the results shown in Proposition~\ref{le:decay_POD}, 
Corollary~\ref{le:decay_prod}, and Proposition~\ref{le:decay_FO}
if we want to study upper error bounds.
\end{rem}

\subsection{Lower Error Bounds}\label{s5:1}

Throughout the study of lower error bounds we assume
the following:
The assumptions (B1) and (B2) with inclusion 
\eqref{inclusion_lower_bounds} are satisfied, and
$C^\lo > 0$ is an upper bound on the 
norm of the embedding of $H(l)$ into $H(1+k)$.

Concerning the weights,
let us first consider the particular case where $\bg \in \M_{d}$. 
The latter property is characterized for product weights in 
Proposition~\ref{l5}. 
For lower bounds we first build the weights 
\[
\br := \bg^{\lo} := T^{\lo}_{d,C^\lo} \bg,
\]
and assuming that $\Z^{\bg^\lo\!,l} \neq \emptyset$ we put
\[
L^\lo := M^{\bg^\lo,l}.
\]

\begin{lemma}\label{l22}
If $\bg \in \M_d$, 
then we have $\Z^{\bg,k} \subseteq \Z^{\bg^\lo,l}$ and
\[
H(L^{\lo}) \sqsubseteq H(K), 
\]
and the corresponding (embedding or restriction) map
$\mathfrak{i}^\lo \colon H(L^\lo) \to H(K)$
has norm at most one. 
\end{lemma}

\begin{proof}
At first we verify that the assumptions (A1)--(A3)
are satisfied with
\[
m := l,\ m^\up := k,\ C := C^\lo,\ \be := \bg^\lo,\ 
\be^\up := T^\up_{d,C^\lo} \bg^\lo.
\]
Due to \eqref{inclusion_lower_bounds} we have (A1), and 
concerning (A2) we note that
$\bg^\lo \in \SSS_{d,C^\lo}$, see 
\eqref{range_T_lo} in Theorem~\ref{t1b} for the non-trivial
case $d=\infty$. Hence we may consider the weights
\[
(\bg^{\lo})^{\up}:= T^{\up}_{d,C^\lo} \bg^{\lo} .
\]
For $d \in \N$ we have $(\bg^{\lo})^{\up} = \bg$, see
Theorem~\ref{t1a}, and for $d = \infty$ we have at least
$(\bg^{\lo})^{\up} \in \M_d$ 
and 
\begin{equation}\label{g96}
(\bg^{\lo})^{\up} \le \bg,
\end{equation}
see 
\eqref{range_T_up} and \eqref{g77} in Theorem~\ref{t1b}. 
The last inequality implies the following. First of all,
$\Z^{\bg,k} \subseteq \Z^{(\bg^{\lo})^{\up},k}$, and since
$\Z^{\bg,k} \neq \emptyset$ by assumption (B1), we have (A3).

Furthermore, we may use Theorems~\ref{Lemma1} and \ref{Lemma4} to
obtain 
\[
H(L^{\lo}) \sqsubseteq H\left(M^{(\bg^{\lo})^{\up},k}\right),
\]
and the corresponding (embedding or restriction) map
$f \mapsto f|_{\Z^{(\bg^{\lo})^{\up},k}}$ has norm at most one.
Moreover,
\[
H\left(M^{(\bg^{\lo})^{\up},k}\right)  \sqsubseteq H(K)
\]
with a corresponding (embedding or restriction) map
$f \mapsto f|_{\Z^{\bg,k}}$ of norm at most one, which follows from
Lemma~\ref{l12} and \eqref{g96}.
This establishes the claim.
\end{proof}

In the sequel, we consider the general case $\bg \in \W_d$ and 
any $\bg_\ast \in \M_{d}$ with $\bg_\ast \leq \bg$. 
For lower bounds we build the weights 
\[
\br := \bg^{\lo}_\ast := T^{\lo}_{d,C^\lo} \bg_\ast
\]
and proceed as before, i.e.,
assuming that $\Z^{\bg^\lo_\ast,l} \neq \emptyset$ we put
\begin{equation}\label{L_lo_ast}
L^\lo_\ast := M^{\bg^\lo_\ast,l}.
\end{equation}

\begin{lemma}\label{l24}
For $\bg \in \W_d$ 
we choose weights $\bg_\ast \in \M_{d}$ with 
$\bg_\ast \leq \bg$
and define $L^\lo_\ast$ as in \eqref{L_lo_ast}.
Then we have $\Z^{\bg,k} \subseteq \Z^{\bg^\lo_\ast,l}$ and
\[
H(L^{\lo}_\ast) \sqsubseteq H(K), 
\]
and the corresponding (embedding or restriction) map
$\mathfrak{i}^\lo \colon H(L^\lo_\ast) \to H(K)$
has norm at most one.
\end{lemma}

\begin{proof}
Since $\bg_{\ast} \le \bg$, we obtain $\Z^{\bg,k} \subseteq
\Z^{\bg_\ast,k}$, and Lemma~\ref{l12} yields
$H(M^{\bg_{\ast},k})  \sqsubseteq H(K)$ with a 
corresponding (embedding or restriction) map
$f\mapsto f|_{\Z^{\bg,k}}$ of norm at most one.
Since $\bg_\ast \in \M_d$, we may apply
Lemma~\ref{l22} to obtain $\Z^{\bg_\ast,k} \subseteq
\Z^{\bg^\lo_\ast,l}$ and
$H(L^{\lo}_\ast) \sqsubseteq H(M^{\bg_\ast,k})$ with
a corresponding (embedding or restriction) map
$f \mapsto f|_{\Z^{\bg_\ast,k}}$ of norm at most one.
This establishes the claim.
\end{proof}

On $H(L^{\lo}_\ast)$ we may consider the continuous linear 
solution operator
\[
S^{\lo} := S \circ \mathfrak{i}^{\lo} \colon
H(L^{\lo}_\ast) \to G\,,\, f \mapsto S(f|_{\Z^{\bg,k}}).
\]
For any linear algorithm $A$ on $H(K)$ we may study the 
linear algorithm
\[
A^{\lo}:= A  \circ \mathfrak{i}^{\lo}
\]
on $H(L^{\lo}_\ast)$. If $A$ is based on function evaluations,
then the same holds true for $A^\lo$.

\begin{corollary}\label{c2}
For $\bg \in \W_d$ 
we choose weights  $\bg_\ast \in \M_{d}$ with 
$\bg_\ast \leq \bg$ and define $L^\lo_\ast$ as in \eqref{L_lo_ast}.
Then we have
\[
\err(A^{\lo}, S^{\lo}, L^{\lo}_\ast) \le \err(A, S, K).
\]
\end{corollary}

Consequently, any lower error bound for $S^{\lo}$ and 
$A^{\lo}$ on $H(L^{\lo}_\ast)$ also holds true for $S$ and  
$A$ on $H(K)$.

\begin{rem}\label{r20}
For $\bg \in \W_d \setminus \M_d$ there seems to be no 
canonical choice of weights 
\[
\bg_\ast \in \Gamma := 
\{ \widetilde{\bg} \in \M_d : \widetilde{\bg} \le \bg \},
\]
but the proof of Lemma~\ref{l24} suggests to choose a `large' 
element of $\Gamma$. In any case the set $\Gamma$ contains a
maximal element, see Appendix~\ref{a3}.
\end{rem}

\appendix

\section{Complete Monotonicity}

In Section~\ref{a1} we show that complete monotonicity of
weights, as introduced in Definition~\ref{def:comp_mon_dec}, is a 
particular instance of a general notion that applies to set functions.
In Section~\ref{a2} we study the case $d \in \N$ and relate complete 
monotonicity of weights to the corresponding notion for functions
on the unit cube $[0,1]^d$.

\subsection{Completely Monotone Set Functions}\label{a1}

Complete monotonicity of non-negative functions is defined and 
studied in \citet{Kurtz74} in the following setting.
Let $\B$ denote the Borel $\sigma$-algebra in a Polish
space $X$ that is locally compact and $\sigma$-compact. For 
$\bg \colon \B \to \R$ and $u, v \in \B$ let
\[
\left(\Delta^\prime_v \bg \right)_u := 
\gamma_{v \cup u} - \gamma_u.
\]
The function $\bg$
is called completely monotone if $\bg \geq \bn$ and
\[
(-1)^n 
\left( \Delta^\prime_{v_n} \dots \Delta^\prime_{v_1} \right) 
\bg \geq \bn 
\]
for every $n \in \N$ and all (non-empty) 
$v_1, \dots, v_n \in \B$; 
the set of all completely monotone functions 
on $\B$ will be denoted by $\M(\B)$. 

In the present paper we consider 
functions $\bg \colon \U_d \to \R$. 
Here the assumptions from \citet{Kurtz74} are met
if $d \in \N$, since $\B = \U_d$ if $X := [d]$ with $d \in \N$
is equipped with the discrete metric, but not
for $d = \infty$. However, for merely the definition of complete 
monotonicity from \citet{Kurtz74} to make sense it suffices to consider 
any set $X$ and any collection $\B$ of subsets of $X$ that is closed
with respect to finite unions. Obviously, the latter holds
true also for $\U_d$ with $d = \infty$.
For functions $\bg \colon \U_d \to \R$ with
$d \in \N$ complete monotonicity is also studied in
\citet{Matus94}.

\begin{lemma}
For $d \in \N \cup \{\infty\}$ let $\bg \colon \U_d \to \R$ and
$v := \{s_1,\dots, s_n\}$ with $n \in \N$ and pairwise 
different $s_i \in [d]$. 
For every $u \in \U_d$ we have 
\begin{equation}\label{g8}
\left( \Delta^\prime_{v} \bg \right)_u
= - \sum_{\ell=1}^n \left( 
\Delta_{\{s_\ell\}} \bg\right)_{u \cup \{s_1,\dots,s_{\ell-1}\}}
\end{equation}
and
\begin{equation}\label{g9}
\left( \Delta_v \bg \right)_u =
(-1)^n 
\left( \Delta^\prime_{\{s_n\}} \dots 
       \Delta^\prime_{\{s_1\}} \bg \right)_u.
\end{equation}
\end{lemma}

\begin{proof}
For the proof of \eqref{g8}
we note that
\[
\left( \Delta^\prime_{v} \bg \right)_u
= \sum_{\ell=1}^n \left( 
\gamma_{u \cup \{s_1,\dots,s_\ell\}}
- \gamma_{u \cup \{s_1,\dots,s_{\ell-1}\}} \right) 
= - \sum_{\ell=1}^n \left( 
\Delta_{\{s_\ell\}} \bg\right)_{u \cup \{s_1,\dots,s_{\ell-1}\}}.
\]
For the proof of \eqref{g9} we note that
\[
\left( \Delta_{\{s_1\}} \bg \right)_u
= \gamma_u - \gamma_{u \cup \{s_1\} } = 
- \left( \Delta^\prime_{\{s_1\}} \bg \right)_u.
\]
Inductively, using Lemma~\ref{l2}, we obtain
\begin{align*}
&\left( \Delta_{\{s_1,\dots,s_{n+1}\}} \bg\right)_u\\
& \quad = 
\left( \Delta_{\{s_1,\dots,s_n\}} \bg\right)_u -
\left( \Delta_{\{s_1,\dots,s_n\}} \bg\right)_{u\cup \{s_{n+1}\}} \\
& \quad = 
(-1)^{n} \cdot \left(
\left( \Delta^\prime_{\{s_{n}\}} \dots 
       \Delta^\prime_{\{s_1\}} \bg \right)_u
-
\left( \Delta^\prime_{\{s_{n}\}} \dots 
       \Delta^\prime_{\{s_1\}} \bg \right)_{u \cup \{s_{n+1}\}}
\right)\\
& \quad =
(-1)^{n+1}
\left( \Delta^\prime_{\{s_{n+1}\}} \dots 
       \Delta^\prime_{\{s_1\}} \bg \right)_u,
\end{align*}
which completes the proof of \eqref{g9}.
\end{proof}

The next lemma shows that $\M(\U_d)$ and $\M_d$, 
which has been introduced in Definition~\ref{def:comp_mon_dec},
coincide.

\begin{lemma}
For every $d \in \N \cup \{\infty\}$ we have $\M(\U_d) = \M_d$.
\end{lemma}

\begin{proof}
The inclusion $\M(\U_d) \subseteq \M_d$ follows immediately from 
\eqref{g9}.

To establish the reverse inclusion we note that
\[
\left(\Delta_w \Delta^\prime_v \bg\right)_u  =
\sum_{\widetilde{w} \subseteq w} (-1)^{| \widetilde{w} |} \left(
\gamma_{v \cup u \cup \widetilde{w}} - \gamma_{u \cup \widetilde{w}}
\right) =
\left(\Delta_w \bg\right)_{v \cup u} -
\left(\Delta_w \bg\right)_{u}
\]
for every $\bg \colon \U_d \to \R$ and all $u,v,w \in \U_d$.
Let $\bg \in \M_d$. Lemma~\ref{lem:diff_op} yields
$\Delta_w \bg \in \M_d$, and therefore 
$\Delta_w \Delta^\prime_v \bg \leq \bn$, see
Lemma~\ref{cmd:claim_i}, which is equivalent to
$-\Delta^\prime_v \bg \in \M_d$. Inductively, we obtain
\[
(-1)^n \left( \Delta^\prime_{v_n} \dots \Delta^\prime_{v_1} \right) \bg 
\in \M_d
\]
for every $n \in \N$ and all $v_1,\dots,v_n \in \U_d$.
Since $\M_d \subseteq \W_d$, we conclude that
$\bg \in \M(\U_d)$.  
\end{proof}

\subsection{Completely Monotone Functions on $[0,1]^d$}\label{a2}

Throughout this section we have $d \in \N$.
At first, we consider real-valued
functions on the continuous unit cube 
$[0,1]^d$. In this context we will use the following 
notation with $w \subseteq [d]$ for the definition of complete
monotonicity.
For a non-empty set $w$ and a point $z\in [0,1]^d$ 
we write $z_w$ to denote the projection 
of $z$ onto those components with indices in $w$; 
we define $z_w$ to be void if $w=\emptyset$. 
Furthermore, we put $-w := [d] \setminus w$.
For two points $z,t\in [0,1]^d$ we write $z_{-w}\!:\!t_w$
to denote the point where 
those components of $z$ with indices in $w$ are replaced by the 
corresponding components of $t$.
Finally, $\psi(z) := (1-z_1,\dots,1-z_d)$ for $z \in [0,1]^d$.

A function $f \colon [0,1]^d \to \R$
is called completely monotonically increasing
if  
\[
\sum_{w \subseteq v} (-1)^{|v|-|w|} f(x_{-w}\!:\!y_w)\ge 0
\]
holds for all non-empty $v\in \U_d$ and 
$x, y\in [0,1]^d$ such that $x_j \leq y_j$ for every $j \in v$.
See, e.g., \citet[Sec.~2.3]{GKOP24}. Moreover,
$f$ will be called completely monotonically decreasing if 
$f \circ \psi$ is completely monotonically increasing.

\begin{lemma}\label{Lemma32}
Let $d\in\N$. A function $f \colon [0,1]^d \to \R$ is 
completely monotonically decreasing if and only if 
\[
\sum_{w \subseteq v} (-1)^{|w|} f(x_{-w}\!:\!y_w)\ge 0
\]
for all non-empty $v\in \U_d$ and 
$x, y\in [0,1]^d$ such that $x_j \leq y_j$ for every $j \in v$.
\end{lemma}

\begin{proof} 
Let $v\in \U_d$ be non-empty and $x, y\in [0,1]^d$.
For $\widetilde{x} := \psi(x)$, $\widetilde{y} := \psi(y)$,
and $\widehat{y} := (\widetilde{x}_{-v}\!:\!\widetilde{y}_v)$
we obtain
\[
(\widetilde{x}_{-(v \setminus w)}\!:\!\widetilde{y}_{v \setminus w})
= 
(\widehat{y}_{-w}\!:\!\widetilde{x}_w)
\]
for every $w \subseteq v$. Therefore
\begin{align*}
\sum_{w \subseteq v} (-1)^{|v|-|w|}
f(\widetilde{x}_{-w}\!:\!\widetilde{y}_w)
&= \sum_{w \subseteq v} (-1)^{|w|}
f(\widetilde{x}_{-(v \setminus w)}\!:\!\widetilde{y}_{v \setminus w}) \\
&= \sum_{w \subseteq v} (-1)^{|w|}
f(\widehat{y}_{-w}\!:\!\widetilde{x}_w).
\end{align*}
For $j \in v$ we have equivalence of
$x_j \leq y_j$ and $\widehat{y}_j \leq \widetilde{x}_j$.
\end{proof}

Note that the concept of complete monotonicity 
is closely linked to the concept of
the variation of a function in the 
sense of Vitali and in the sense of Hardy and Krause. Indeed, 
these concepts involve sums of differences 
of function values that can be shown to be all non-negative for 
completely monotonically increasing functions, and all non-positive 
for completely monotonically decreasing functions.
See \cite{AD2015} for details.

Of course, every 
function $\bg \colon \U_d \to \R$
may be identified in a canonical way with a 
function $f \colon \{0,1\}^d \to \R$.
More precisely, let $\varphi \colon \{0,1\}^d \to \U_d$ be given by 
\[
\varphi(x) := \{ j \in [d] : x_j = 1\}.
\]
We obtain a bijection $\Phi$ between 
the set of all real-valued functions on $\U_d$
and the set of all real-valued functions on $\{0,1\}^d$
by
\[
(\Phi \bg) (x) := \gamma_{\varphi(x)}
\]
for $x \in \{0,1\}^d$ and 
$\bg \colon \U_d \to \R$. 

The next lemma shows that
$\M_d$ coincides with the set of all restrictions of non-negative,
completely monotonically decreasing functions 
to the Hamming cube $\{0,1\}^d$,
up to the identification via $\Phi$.

\begin{lemma}\label{Lemma33}
Let $d\in \N$.
The mapping $\Phi$ defines a bijection from $\M_d$ onto 
\[
\M_d^\prime := \{ f|_{\{0,1\}^d} : 
\text{$f \colon [0,1]^d \to {[0,\infty)}$ completely monotonically
decreasing}\}. 
\]
\end{lemma}

\begin{proof}
For $f \colon [0,1]^d \to {[0,\infty)}$ let $\bg \in \W_d$ be
defined by 
\begin{equation}\label{g88}
\Phi \bg  = f|_{\{0,1\}^d}. 
\end{equation}
Furthermore, let $u,v \subseteq [d]$ with $u \cap v = \emptyset$ and
let $x,y \in \{0,1\}^d$ be defined by
$\varphi(x) = u$ and $\varphi(y) = v$.
Note that $u \cup w = \varphi(x_{-w}\!:\!y_w)$ for every $w \subseteq v$. 
Thus we obtain
\begin{align}\label{function_c_m_d_klaus}
\left(\Delta_v \bg \right)_u 
&= 
\sum_{w\subseteq v} (-1)^{|w|} \gamma_{u\cup w} =
\sum_{w\subseteq v} (-1)^{|w|} \gamma_{\varphi(x_{-w}:y_w)}
\nonumber\\
&=
\sum_{w\subseteq v} (-1)^{|w|} f(x_{-w}\!:\!y_w).
\end{align}
Moreover, $x_j = 0 <1 = y_j$ for $j \in v$. 
Due to Lemma~\ref{Lemma32}, $\bg \in \M_d$ 
if $f$ is completely monotonically decreasing, i.e., 
$\M_d^\prime \subseteq \Phi(\M_d)$.

For the proof of the reverse inclusion we extend the mapping
$\varphi$ from $\{0,1\}^d$ to $[0,1]^d$ by
\[
\varphi(x) := \{ j \in [d] : x_j > 0\}.
\]
For $\bg \in \M_d$ let $f \colon [0,1]^d \to {[0,\infty)}$
be defined by 
\[
f(x) := \gamma_{\varphi(x)}
\]
for every $x \in [0,1]^d$. Obviously, we have \eqref{g88}. Thus it
remains to show that $f$ is completely monotonically decreasing.

For $x \in [0,1]^d$ we put 
$\lceil x \rceil := (\lceil x_1 \rceil, \dots, \lceil x_d \rceil) 
\in \{0,1\}^d$. Clearly $\varphi(x) = \varphi(\lceil x \rceil)$,
and therefore $f(x) = f(\lceil x \rceil)$. It follows that $f$ is 
completely monotonically decreasing, if 
\[
\alpha := \sum_{w \subseteq v} (-1)^{|w|} f(x_{-w}\!:\!y_w) \geq 0
\]
for all non-empty $v\in \U_d$ and $x, y\in \{0,1\}^d$ 
such that $x_j \leq y_j$ for every $j \in v$, see Lemma~\ref{Lemma32}.
For $v,x,y$ chosen in this way we put
\[
v_1 := \{j \in v : x_j < y_j\} = \{j \in v : x_j = 0,\ y_j=1\}
\]
and $v_2 := v \setminus v_1$.
Let $w_i \subseteq v_i$ for $1\le i\le 2$. We have
\[
(x_{-(w_1 \cup w_2)}\!:\! y_{w_1 \cup w_2}) = (x_{-w_1}\!:\!y_{w_1}),
\]
which implies
\begin{align*}
\alpha &=
\sum_{w_1 \subseteq v_1} \sum_{w_2 \subseteq v_2}
(-1)^{|w_1|+|w_2|} f(x_{-(w_1\cup w_2)}\!:\!y_{w_1\cup w_2})\\
&=
\sum_{w_1 \subseteq v_1} (-1)^{|w_1|} f(x_{-w_1}\!:\!y_{w_1})
\sum_{w_2 \subseteq v_2} (-1)^{|w_2|}.
\end{align*}
Therefore $\alpha = 0$ if $v_2 \neq \emptyset$, so that
it remains to consider the case $v_2 = \emptyset$, i.e., $v = v_1$.
Without loss of generality we may assume that $y_j = 0$ for
$j \not\in v$, so that $v = \varphi(y)$. Let $u := \varphi(x)$.
Since $u \cap v = \emptyset$, the first part of the proof is
applicable, and \eqref{function_c_m_d_klaus} together with $\bg \in
\M_d$ yields $\alpha \geq 0$. This concludes the proof of
$\Phi(\M_d) \subseteq \M_d^\prime$.
\end{proof}

\section{Finite Measures, CDFs and Densities on $\U_d$}
\label{r8}

The major findings from Theorems~\ref{t1a} and \ref{t1b} may be
reformulated in terms of finite measures, densities (w.r.t.\ the
counting measure), and cumulative distribution functions (CDFs) as 
follows.

Obviously, the finite measures 
$\mu$ on the power set of $\U_d$ may be identified with the
elements $\bg \in \SSS_{d,1}$ via
$\gamma_u := \mu(\{u\})$ and $\mu(A) := \sum_{w \in A} \gamma_w$
for $u \in \U_d$ and $A \subseteq \U_d$. 
That is, $\bg$ is the density of $\mu$ (w.r.t. the counting measure).
For $\bg \in \SSS_{d,1}$
and $\bg^\up := T^\up_{d,1} \bg$ we have
\[
\gamma^\up_u =
\mu(\{ w \in \U_d : u \subseteq w\})
\]
by definition, so that $\bg^\up$
may be considered as the CDF associated to $\mu$. 
We have $T^\lo_{d,1} \bg^\up = \bg$.
If we interpret $T^\lo_{d,1} \bg^\up$ as the derivative of $\bg^\up$,
then we may conclude that the
derivative of the CDF is the density of $\mu$, 
a kind of result which is well-known in the continuous case.
Moreover,
\[
\left(\Delta_v \bg^\up \right)_u =  
\mu(\{ w \in \U_d : u \subseteq w \wedge v \cap w = \emptyset\}),
\]
which follows by induction on $|v|$ using Lemma~\ref{l2}. 
For $s,r \in \N$ 
let 
\[
A(s,r) := \{ w \in \U_d : [s] \setminus [r] \cap w = \emptyset\}.
\]
Since 
\[
\bigcup_{r \in \N} \bigcap_{s \in \N} A(s,r) =
\bigcup_{r \in \N} \{ w \in \U_d : \max w \leq r\} = \U_d,
\]
the $\sigma$-continuity of $\mu$ yields
\[
\gamma^\up_u = \mu
\left(
\bigcup_{r \in \N} \bigcap_{s \in \N} \{ w \in A(s,r) : u \subseteq w\} 
\right) 
=
\lim_{r \to \infty} \lim_{s \to \infty}
\left( \Delta_{[s] \setminus [r]} \bg^\up \right)_u
\]
for every $u \in \U_d$, which is the continuity property required
in the definition of $\A_d$,
cf.\ \eqref{def_A_d}.

Conversely, let $\bg \in T^{\up}_{d,1} (\SSS_{d,1})$.
For $\bg^\lo := T^{\lo}_{d,1} \bg \in \SSS_{d,1}$ we have
$T^{\up}_{d,1} \bg^\lo = \bg$,
so that $\bg$ is the CDF for the measure with density 
$\bg^\lo$.

\section{Maximal Completely Monotone Weights}\label{a3}

In this section we assume that 
$\bg \in \W_d \setminus \M_d$. In this
case lower error bounds have been obtained in Section~\ref{s5:1} by
means of weights of the form $T^\lo_{d,C^\lo} \bg_\ast$ with any choice of
$\bg_\ast \in \Gamma 
= \{ \widetilde{\bg} \in \M_d : \widetilde{\bg} \le \bg \}$, 
cf.\ Remark~\ref{r20}.

\begin{lemma}\label{Lemma:Zorn}
For every $\bg \in \W_d \setminus \M_d$ the set
$\Gamma$ contains a maximal element.
\end{lemma}

\begin{proof}
Since the trivial weights ${\bf 0}$ are 
in $\Gamma$, the set is not empty. 
Endowed with the relation $\le$ the set $\Gamma$ is partially 
ordered. Let $I$ be some index set, and let 
$(\bg^{(i)})_{i\in I}$ be a totally ordered family in $\Gamma$. 
Then we may define $\widehat{\bg} \in \W_d$ via 
\[
\widehat{\gamma}_u := \sup \{\gamma^{(i)}_u : i \in I\}
\]
for all $u\in \U_d$. Note that 
$0 \le \widehat{\gamma}_u \le \gamma_u$ for each $u\in \U_d$, 
due to the definition of $\Gamma$. We now show that 
$\widehat{\bg}$ is completely monotone. Let $u, v\in \U_d$ and 
$\eps >0$ be arbitrary. Due to the definition of the supremum, 
we find for each $w\subseteq v$ some index $i_w\in I$ such that
\[
\widehat{\gamma}_{u\cup w} - \gamma^{(i_w)}_{u\cup w} 
\le \eps / 2^{|v|}.
\]
Since $(\bg^{(i)})_{i\in I}$ is totally ordered 
and $\{i_w : w\subseteq v\}$ is a finite subset of $I$, there 
exists an index $j \in \{i_w : w\subseteq v\}$ 
such that
\[
\bg^{(i_w)} \leq \bg^{(j)}
\]
for every $w \subseteq v$. In particular,
\[
\gamma^{(i_w)}_{u\cup w} \le 
\gamma^{(j)}_{u\cup w} \le \widehat{\gamma}_{u\cup w} 
\]
for all $w\subseteq v$. Therefore we have
\begin{equation*}
\begin{split}
\left( \Delta_v \widehat{\bg} \right)_u 
&= \sum_{w\subseteq v} (-1)^{|w|} \widehat{\gamma}_{u\cup w}
=  \sum_{w\subseteq v} (-1)^{|w|} \left(   \gamma^{(j)}_{u\cup w} 
+ \left( \widehat{\gamma}_{u\cup w} -  
\gamma^{(j)}_{u\cup w} \right) \right)\\
&\ge \sum_{w \subseteq v} (-1)^{|w|}  \gamma^{(j)}_{u\cup w} -  
\sum_{w \subseteq v}  
 \left( \widehat{\gamma}_{u\cup w} -  \gamma^{(j)}_{u\cup w} \right) 
\ge \left( \Delta_v \gamma^{(j)} \right)_u - \eps \\
&\ge  - \eps,
\end{split}
\end{equation*}
where in the last inequality we used that $\bg^{(j)}$ is completely 
monotone.
Letting $\eps$ tend to zero, we see that 
$\left( \Delta_v \widehat{\bg} \right)_u \ge 0$. Thus 
$\widehat{\bg}$ is completely monotone and therefore an 
upper bound of
$(\bg^{(i)})_{i\in I}$ in $\Gamma $. Due to Zorn's lemma the 
set $\Gamma $ has a maximal element $\bg^{\ast}$. 
\end{proof}

Considering Lemma \ref{Lemma:Zorn}, it is a natural question whether 
$\Gamma$ contains a greatest element.
As we shall see, this is in general not the case. 

\begin{exmp}
Let $d := 2$ and $C := 1$. We consider the sum operator
$T:=T_{2,1}^\up$ as a linear mapping on $\R^{\U_2}$, i.e.,
\[
\left(T \be\right)_u :=
\sum_{\substack{v\in \U_2 \\ u\subseteq v}} \eta_v
\]
for $\be \in \R^{\U_2}$.
Obviously $T$ is a bijection, and due to Theorem~\ref{t1a} we
have $\be \in \M_2$ if and only if $T^{-1} \be \geq 0$.

As a counter-example,
we consider the weights $\bg \in \W_2$ with 
\[
\gamma_{\emptyset}=5,\quad 
\gamma_{\{1\}}=5,\quad \gamma_{\{2\}}=3,\quad \gamma_{\{1,2\}}=1.
\]
It is easily checked that 
\[
\left(T^{-1}\bg\right)_{\emptyset}=-2,\quad 
\left(T^{-1}\bg\right)_{\{1\}}=4,\quad 
\left(T^{-1}\bg\right)_{\{2\}}=2,\quad 
\left(T^{-1}\bg\right)_{\{1,2\}}=1,
\]
and therefore $\bg\not\in\M_2$. 

Now consider $\be^{(\eps)} \in \W_2$ given by 
\[
\eta_{\emptyset}^{(\eps)}=5,\quad \eta_{\{1\}}^{(\eps)}=5,\quad 
\eta_{\{2\}}^{(\eps)}=1+\eps,\quad \eta_{\{1,2\}}^{(\eps)}=1,
\]
for $\eps \geq 0$. Observe that
$\be^{(0)} \leq \widetilde{\bg} \leq \bg$ is equivalent to 
$\widetilde{\bg} = \be^{(\eps)}$ with $0 \leq \eps \leq 2$.
It is easily checked that 
\begin{gather*}
\left(T^{-1}\be^{(\eps)}\right)_{\emptyset}=-\eps,\quad 
\left(T^{-1}\be^{(\eps)}\right)_{\{1\}}=4,\\
\left(T^{-1}\be^{(\eps)}\right)_{\{2\}}=\eps,\quad 
\left(T^{-1}\be^{(\eps)}\right)_{\{1,2\}}=1,
\end{gather*}
and therefore $\be^{(\eps)} \in\M_2$ if and only if $\eps = 0$.
We conclude that $\be^{(0)}$ is a maximal element of $\Gamma$.

In the next step, consider 
$\bz^{(\eps)} \in \W_2$ given by 
\[
\zeta_{\emptyset}^{(\eps)}=5,\quad \zeta_{\{1\}}^{(\eps)}=3+\eps,\quad 
\zeta_{\{2\}}^{(\eps)}=3,\quad \zeta_{\{1,2\}}^{(\eps)}=1,
\]
for $\eps \geq 0$.
It is easily checked that 
\begin{gather*}
\left(T^{-1}\bz^{(\eps)}\right)_{\emptyset}=-\eps,\quad 
\left(T^{-1}\bz^{(\eps)}\right)_{\{1\}}=2+\eps,\\
\left(T^{-1}\bz^{(\eps)}\right)_{\{2\}}=2,\quad 
\left(T^{-1}\bz^{(\eps)}\right)_{\{1,2\}}=1.
\end{gather*}
As before, we conclude that $\bz^{(0)}$ is a maximal element 
of $\Gamma$. Since $\be^{(0)} \neq \bz^{(0)}$, the set $\Gamma$
does not contain a greatest element.
\end{exmp}

\subsection*{Acknowledgment}

We are grateful to Jan Loran and Carsten Steiber
for comments on earlier versions of this paper.
Part of this work was done, while M.G. and K.R. have visited 
the RICAM in Linz and the Mathematisches Seminar in Kiel; we are
grateful for the kind hospitality. The second author acknowledges the 
support of the Austrian Science Fund (FWF) Project  
P 34808/Grant DOI: 10.55776/P34808. For open access purposes, the authors 
have applied a CC BY public copyright license to any author accepted 
manuscript version arising from this submission.

\footnotesize

\bibliographystyle{abbrvnat}

\end{document}